\documentclass[11pt,reqno]{amsart}

\usepackage{color}
\usepackage[colorlinks=true, allcolors=blue]{hyperref}
\usepackage{amsmath, amssymb, amsthm}
\usepackage{mathrsfs}
\usepackage{mathtools}
\usepackage[noabbrev,capitalize,nameinlink]{cleveref}
\crefname{equation}{}{}
\usepackage{fullpage}
\usepackage[noadjust]{cite}
\usepackage{graphics}
\usepackage{pifont}
\usepackage{tikz}
\usepackage{bbm}
\usepackage[T1]{fontenc}

\usetikzlibrary{arrows.meta}

\usepackage{environ}
\usepackage{framed}
\usepackage{url}
\usepackage[linesnumbered,ruled,vlined]{algorithm2e}
\usepackage[noend]{algpseudocode}
\usepackage[labelfont=bf]{caption}
\usepackage{cite}
\usepackage{framed}
\usepackage[framemethod=tikz]{mdframed}
\usepackage{appendix}
\usepackage{graphicx}
\usepackage[textsize=tiny]{todonotes}
\usepackage{tcolorbox}
\allowdisplaybreaks[1]

\crefname{algocf}{Algorithm}{Algorithms}

\crefname{equation}{}{} 
\AtBeginEnvironment{appendices}{\crefalias{section}{appendix}} 

\usepackage{enumerate}
\usepackage[color,final]{showkeys} 

\colorlet{refkey}{orange!20}
\colorlet{labelkey}{blue!30}

\crefname{algocf}{Algorithm}{Algorithms}

\numberwithin{equation}{section}
\newtheorem{theorem}{Theorem}[section]
\newtheorem{proposition}[theorem]{Proposition}
\newtheorem{lemma}[theorem]{Lemma}

\crefname{claim}{Claim}{Claims}

\newtheorem{corollary}[theorem]{Corollary}
\newtheorem{conjecture}[theorem]{Conjecture}
\newtheorem*{question*}{Question}

\theoremstyle{definition}
\newtheorem{definition}[theorem]{Definition}

\newtheorem*{definition*}{Definition}

\theoremstyle{remark}
\newtheorem*{remark}{Remark}

\newcommand{\snorm}[1]{\lVert#1\rVert}

\newcommand{\sang}[1]{\langle #1 \rangle}

\newcommand{\mb}{\mathbb}

\newcommand{\mbm}{\mathbbm}
\newcommand{\mc}{\mathcal}

\newcommand{\ol}{\overline}
\newcommand{\on}{\operatorname}

\newcommand{\wt}{\widetilde}

\allowdisplaybreaks
\title{Sharp invertibility of random Bernoulli matrices}

\author[A1]{Vishesh Jain}
\address{Department of Statistics, Stanford University, Stanford, CA 94305 USA}
\email{visheshj@stanford.edu}

\author[A2]{Ashwin Sah}
\author[A3]{Mehtaab Sawhney}
\address{Department of Mathematics, Massachusetts Institute of Technology, Cambridge, MA 02139, USA}
\email{\{asah,msawhney\}@mit.edu}

\begin{document}

\begin{abstract}
Let $p \in (0,1/2)$ be fixed, and let $B_n(p)$ be an $n\times n$ random matrix with i.i.d.~Bernoulli random variables with mean $p$. We show that for all $t \ge 0$,
\[\mb{P}[s_n(B_n(p)) \le tn^{-1/2}] \le C_p t + 2n(1-p)^{n} + C_p (1-p-\epsilon_p)^{n},\]
where $s_n(B_n(p))$ denotes the least singular value of $B_n(p)$ and $C_p, \epsilon_p > 0$ are constants depending only on $p$. In particular, 
\[\mb{P}[B_{n}(p) \text{ is singular}] = 2n(1-p)^{n} + C_{p}(1-p-\epsilon_p)^{n},\]
which confirms a conjecture of Litvak and Tikhomirov. 

We also confirm a conjecture of Nguyen by showing that if $Q_{n}$ is an $n\times n$ random matrix with independent rows that are uniformly distributed on the central slice of $\{0,1\}^{n}$, then 
\[\mb{P}[Q_{n} \text{ is singular}] = (1/2 + o_n(1))^{n}.\]
This provides, for the first time, a sharp determination of the logarithm of the probability of singularity in any natural model of random discrete matrices with dependent entries.  
\end{abstract}

\maketitle

\section{Introduction}\label{sec:introduction}

The singularity of discrete random matrices is one of the fundamental problems of the non-asymptotic theory of random matrices. Let $B_n(p)$ denote an $n\times n$ random matrix with independent and identically distributed (i.i.d.) $\on{Ber}(p)$ entries (i.e.~each entry is independently $1$ with probability $p$ and $0$ with probability $1-p$). Let $q_n(p) := \mb{P}[B_n(p)\text{ is singular}]$. A classical result of Koml\'os \cite{Kom67} shows that $q_n(1/2) = o_n(1)$. Nearly 30 years later, the bound $q_n(1/2) \le 0.999^{n}$ was obtained by Kahn, Koml\'os, and Szemer\'edi \cite{KKS95}. The base of the exponent was improved to $0.939$ by Tao and Vu \cite{TV06}, $3/4 + o_n(1)$ also by Tao and Vu \cite{TV07}, and $1/\sqrt{2} + o_n(1)$ by Bourgain, Vu, and Wood \cite{BVW10}. In 2018, a breakthrough result of Tikhomirov \cite{Tik20} established that $q_n(p) = (1-p+o_n(1))^{n}$ for fixed $p \in (0,1/2]$.

It is widely believed that the dominant reason for the singularity of discrete random matrices is, in a certain sense, local. A natural extension of a folklore conjecture (cf.~\cite{BVW10, KKS95, LT20, VU14}) in this direction is the following. 
\begin{conjecture}
\label{conjecture:invertibility-bernoulli}
For fixed $p \in (0,1)$,
\[q_n(p) = (1+o_n(1))\mb{P}[B_n(p)\emph{ has two equal rows, two equal columns, a zero row, or a zero column}].\]
\end{conjecture}
Note that this conjecture falls outside the scope of the aforementioned result of Tikhomirov \cite{Tik20}, since an analysis of the error term $(1+o_n(1))^{n}$ appearing there reveals that it grows superpolynomially in $n$ (cf.~\cite{LT20}). 
Since for $p \in (0,1/2)$, the probability that a row or column is zero is exponentially more likely than the probability that two rows or two columns are equal, the following is a slight generalization of \cref{conjecture:invertibility-bernoulli} in the case when $p \in (0,1/2)$, which is stated in a recent work of Litvak and Tikhomirov \cite{LT20}.   
\begin{conjecture}[cf.~{\cite[Problem~8.2]{LT20}}]
\label{conjecture:invertibility-sparse}
Let $p_n$ be a sequence of real numbers satisfying $0<\liminf p_n \le \limsup p_n < 1/2$. Then
\[q_n(p_n) = (1+o_n(1))\mb{P}[B_n(p)\emph{ has a zero row or column}]= (2+o_n(1))n(1-p_n)^n.\]
\end{conjecture}

In recent work \cite{LT20}, Litvak and Tikhomirov confirmed the validity of a sparse version of \cref{conjecture:invertibility-sparse}. More precisely, they proved that the stated conclusion holds for all sequences of real numbers $p_n$ satisfying $Cn^{-1}\ln{n}\le p_n \le c$, where $c>0$ is a small universal constant, and $C > 0$ is a large universal constant. We note that in a certain range of sparsity limited to $ p_n \le n^{-1}\ln n + o_n(n^{-1}\ln\ln n)$, an analogous result was obtained in earlier work of Basak and Rudelson \cite{BR18}. Very recently, using an intricate analysis of so-called `steep vectors', Huang \cite{Hua20} was able to bridge the gap between the regimes covered by \cite{BR18} and \cite{LT20}, thereby completing the program of establishing \cref{conjecture:invertibility-sparse} in all sparse regimes of interest. 

In this paper, we establish a framework for addressing sharp questions about the singularity of random discrete matrices. As our first application, we resolve \cref{conjecture:invertibility-sparse}. 

\begin{theorem}\label{thm:main}
Fix $p \in (0,1/2)$. There exist constants $C_{p}, \epsilon_{p}, n_{p} > 0$ such that for all $n \ge n_{p}$ and $t\ge 0$,
\[\mb{P}[s_n(B_n(p))\le t/\sqrt{n}]\le C_{p} t + (2+(1-\epsilon_{p})^n)n(1-p)^n.\]
\end{theorem}
\begin{remark}
It is readily seen from the proof that for any $\alpha \in (0,1/4)$, we may choose $\epsilon_p = \epsilon'_\alpha$, $n_p = n'_\alpha$, and $C_p = C'_\alpha$ for all $p \in (\alpha, 1/2 - \alpha)$. This uniformity in the choice of parameters is needed to resolve \cref{conjecture:invertibility-sparse}.
\end{remark}

Here, $s_n(B)$ denotes the least singular value of the $n\times n$ matrix $B$, which may be characterized as $s_n(B) = \inf_{x\in \mb{S}^{n-1}}\snorm{Bx}_{2}$, where $\mb{S}^{n-1}$ denotes the unit sphere in $\mb{R}^{n}$ and $\snorm{\cdot}_{2}$ denotes the standard Euclidean norm on $\mb{R}^{n}$.\\ 

The techniques developed in this paper also lend themselves to the solution of the following problem, which has received attention from several researchers \cite{Ngu13, FJLS19, Jai19, Tra20} in recent years as a test bed for (inverse) Littlewood--Offord type results in dependent settings. Let $Q_{n}$ denote a random matrix with independent rows, each of which is chosen uniformly from among those vectors in $\{0,1\}^{n}$ which have sum exactly $\lfloor n/2 \rfloor$. The study of this model was initiated by Nguyen in \cite{Ngu13} as a natural relaxation of the adjacency matrix of a uniformly random (directed) regular graph. He conjectured the following. 

\begin{conjecture}[{\cite[Conjecture~1.4]{Ngu13}}]
\label{conj:combinatorial}
With notation as above,
\[\mb{P}[Q_n \emph{ is singular}] = \left(\frac{1}{2} + o_n(1)\right)^{n}.\]
\end{conjecture}
Nguyen showed that $\mb{P}[Q_n \text{ is singular}] = O_{C}(n^{-C})$ for any $C > 0$. After intermediate work \cite{FJLS19, Jai19}, an exponential upper bound on the singularity probability was only very recently obtained in work of Tran \cite{Tra20}. Our next result settles \cref{conj:combinatorial}, using completely different techniques compared to \cite{Ngu13, FJLS19, Jai19, Tra20}.

\begin{theorem}
\label{thm:row-regular}
For every $\epsilon>0$, there exists $C_{\epsilon}$ depending on $\epsilon$ such that for all sufficiently large $n$, and for all $t\ge 0$,  
\[\mb{P}[s_n(Q_n)\le t/\sqrt{n}]\le C_\epsilon t + (1/2+\epsilon)^n.\]
\end{theorem}
\begin{remark}
To the best of our knowledge, this is the first instance where $\lim_{n \to \infty}n^{-1}\cdot \log\mb{P}[M_n \text{ is singular}]$ has been determined for any natural model of random discrete matrices $M_n$ with non-trivially dependent entries. 
\end{remark}


\subsection{Overview of techniques}
\label{sub:overview}
Since the seminal works of Litvak, Pajor, Tomczak-Jaegermann, and Rudelson \cite{LPRT05}, Rudelson \cite{Rud08}, and Rudelson and Vershynin \cite{RV08}, upper bounds on the lower tail of the least singular value have been established by decomposing the unit sphere into `structured' and `unstructured' vectors, and analyzing invertibility on the two components separately. A high-level strategy of this nature underpins all recent sharp results on the invertibility of discrete random matrices \cite{BR18, LT20}, and will also be used in our work. However, our definition and treatment of both the `structured' and the `unstructured' components departs significantly from \cite{BR18, LT20}, as we now briefly discuss.

\emph{Structured vectors: }The structured vectors in our work are the so-called almost-constant vectors, which are those vectors on the unit sphere which have a $(1-\delta)$ fraction of their coordinates within $\rho/\sqrt{n}$ of each other, where $\delta, \rho > 0$ are small constants. This class of structured vectors arises naturally in the consideration of the invertibility of random matrices with prescribed row (and/or column) sums (cf.~\cite{JSS20digraphs, LLTTY17, Tra20}); the connection to our setting, where all the entries of the matrix are independent, will be evident from the discussion of unstructured vectors. We remark that in \cite{LT20} a different and much more involved class of structured vectors is considered using different techniques.

Invertibility on almost-constant vectors has been considered previously in \cite{JSS20digraphs, LLTTY17, Tra20}. However, for the purpose of proving the sharp invertibility estimate in \cref{thm:main}, the techniques in these works are insufficient. In \cref{sec:compressible}, we develop a new technique to prove sharp invertibility on the set of almost-constant vectors. Roughly speaking, the only problematic almost-constant vectors are those which are very close to a standard basis vector (indeed, standard basis vectors correspond to the event that the matrix has a zero column). For concreteness, consider vectors which are very close to the first standard basis vector. We show that, if any such vector has exponentially small image under $B_n(p)$, then either the first column of $B_n(p)$ is the zero vector, or it must belong to a universal subset of nonzero vectors of $\{0,1\}^{n}$ of measure (under $\on{Ber}(p)^{\otimes n}$) at most $(1-p+\epsilon)^{n}$. The first case corresponds to the term $(1-p)^{n}$; for the second case, we use the work of Rudelson and Vershynin to show that, on our event, the probability that any vector in this universal subset appears as the first column of $B_n(p)$ is at most $\exp(-4\epsilon n)$, at which point we can conclude using the union bound.

\emph{Unstructured vectors: }The starting point of our treatment of unstructured vectors is the insight of Litvak and Tikhomirov \cite{LT20} to exploit the exponential gap (for $p < 1/2$) between $(1-p)^{n}$ and $\binom{n}{pn}^{-1}$ by considering, for $x \in \mb{S}^{n-1}$ an unstructured vector, the anticoncentration of the random sum $b_1 x_1 + \dots + b_n x_n$, where $(b_1,\dots, b_n)$ is a random $\{0,1\}^{n}$ vector \emph{conditioned to have sum close to} $pn$. However, in order to do this, one needs a rather precise theory of anticoncentration on a slice of the Boolean hypercube, which has traditionally been quite challenging. For instance, despite much interest (cf.~\cite{Cook17b, FJLS19, Jai19, LLTTY17, Ngu13}), exponential singularity bounds for $\{0,1\}$-valued random matrices with constrained row sums \cite{Tra20} (respectively, row and column sums \cite{JSS20digraphs}) were only very recently established.

We note that both the works \cite{JSS20digraphs, Tra20} as well as \cite{LT20} use a proxy for the conditional anticoncentration function resulting from manipulating Esseen's inequality (these are analogues of the essential least common denominator of Rudelson and Vershynin \cite{RV08} for the slice). In the case of \cite{JSS20digraphs, Tra20}, the modification is closer in spirit to \cite{RV08}, but cannot give sharp results, whereas the version in \cite{LT20} is significantly more involved, but is not powerful enough to handle slices of the Boolean hypercube which are not very far from the central slice.

In this work, we completely avoid the use of Esseen's inequality, choosing instead to work directly with the conditional anticoncentration function. Our approach is very much influenced by \cite{Tik20}, where a similar strategy was employed to show that the probability of singularity of $B_n(p)$ is at most $(1-p + o_n(1))^{n}$. Indeed, our key ingredient (\cref{thm:inversion-of-randomness}) shows that the conditional anticoncentration of uniformly random points of certain generic lattices is sufficiently good, and is an analogue of \cite[Theorem~B]{Tik20}, with the dependence of random variables on the slice requiring a necessarily delicate treatment.\\

Finally, for the proof of \cref{thm:row-regular}, since our treatment of unstructured vectors (discussed above) is already for the anti-concentration with respect to a uniformly random element of a Boolean slice, it only remains to analyse structured vectors. Compared to \cref{thm:main}, we need a less precise estimate and hence, no longer need to isolate `elementary' vectors. However, an additional complication arises since the inner product of any row of $Q_n$ with the all-ones vector is deterministically $\lfloor n/2 \rfloor$. This requires decomposing structured vectors into those which have sufficient mass in the all-ones direction and those which do not, and then using a suitable argument for each case.

\subsection{Extensions}
\label{sub:second-paper}
The ideas used to establish \cref{thm:inversion-of-randomness}, as well as our treatment of structured vectors, turn out to be remarkably robust. Indeed, the present article should be viewed as a companion work to \cite{JSS20discrete2}, which proves sharp invertibility results for random matrices whose individual entries are identical copies of a fixed, discrete random variable which is not uniform on its support, as well as logarithmically-optimal invertibility estimates (as in \cite{Tik20}) for all discrete distributions. Although the results of \cite{JSS20discrete2} subsume \cref{thm:main} in this work, we have chosen to isolate the sparse Bernoulli case due to the significant interest in it (cf.~\cite[Problem~8.2]{LT20}, \cite{BR18, Hua20}) and because the simplified setting of Bernoulli random variables allows us to present the basic elements of our framework in a clearer fashion. Additionally, the present work contains \cref{thm:row-regular} as well as complete details of certain arguments which are only sketched in \cite{JSS20discrete2}.

\subsection{Notation}\label{sub:notation}
For a positive integer $N$, $\mb{S}^{N-1}$ denotes the set of unit vectors in $\mb{R}^{N}$, and if $x\in\mb{R}^N$ and $r\ge 0$ then $\mb{B}_2^N(x,r)$ denotes the radius $r$ Euclidean ball in $\mb{R}^{N}$ centered at $x$. $\snorm{\cdot}_2$ denotes the standard Euclidean norm of a vector, and for a matrix $A = (a_{ij})$, $\snorm{A}$ is its spectral norm (i.e., $\ell^{2} \to \ell^{2}$ operator norm).

For nonnegative integers $m\le n$, we let $\{0,1\}^n_m$ be the set of vectors in $\{0,1\}^{n}$ with sum $m$. We will let $[N]$ denote the interval $\{1,\dots, N\}$.  

We will also make use of asymptotic notation. For functions $f,g$, $f = O_\alpha(g)$ (or $f\lesssim_\alpha g$) means that $f \le C_\alpha g$, where $C_\alpha$ is some constant depending on $\alpha$; $f = \Omega_\alpha(g)$ (or $f \gtrsim_\alpha g$) means that $f \ge c_\alpha g$, where $c_\alpha > 0$ is some constant depending on $\alpha$, and $f = \Theta_\alpha(g)$ means that both $f = O_\alpha(g)$ and $f = \Omega_\alpha(g)$ hold. For parameters $\epsilon, \delta$, we write $\epsilon \ll \delta$ to mean that $\epsilon \le c(\delta)$ for a sufficient function $c$. 

For a random variable $\xi$ and a real number $r \ge 0$, we let $\mc{L}(\xi, r):= \sup_{z \in \mb{R}}\mb{P}[|\xi-z|\le r]$. We will use $\ell_1(\mb{Z})$ to denote the set of functions $f\colon \mb{Z} \to \mb{R}$ for which $\sum_{z \in \mb{Z}}|f(z)| < \infty$. 
Finally, we will omit floors and ceilings where they make no essential difference.

\subsection{Acknowledgements}
We thank Mark Rudelson, Konstantin Tikhomirov, and Yufei Zhao for comments on the manuscript. A.S.~and M.S.~were supported by the National Science Foundation Graduate Research Fellowship under Grant No.~1745302. This work was done when V.J.~was participating in a program at the Simons Institute for the Theory of Computing.

\section{Invertibility on almost-constant vectors}\label{sec:compressible}

\subsection{Almost-constant and almost-elementary vectors}\label{sub:almost-vectors}
We recall the usual notion of almost-constant vectors, a modification of compressible vectors (see, e.g., \cite{Tra20}).
\begin{definition}[Almost-constant vectors]
\label{def:almost-constant}
For $\delta, \rho \in (0,1)$, we define 
$\on{Cons}(\delta,\rho)$ to be the set of $x\in\mb{S}^{n-1}$ for which there exists some $\lambda\in\mb{R}$ such that $|x_i-\lambda|\le\rho/\sqrt{n}$ for at least $(1-\delta)n$ values $i\in[n]$.
\end{definition}

The main result of this section is the following. 

\begin{proposition}\label{prop:compressible}
For any $c > 0$, there exist $\delta,\rho,\epsilon, n_0 > 0$ depending only on $c$, such that for all $n \ge n_0$, 
\[\mb{P}[\exists x\in\on{Cons}(\delta,\rho): \snorm{B_n(p)x}_2\le 2^{-cn}]\le n(1-p)^n + (1-p-\epsilon)^n.\]
\end{proposition}

For later use, we record the following simple property of non-almost-constant vectors. 

\begin{lemma}
\label{lem:nonconstant-admissible}
For $\delta,\rho \in (0,1/4)$, there exist $\nu,\nu' > 0$ depending only on $\delta, \rho$, and a finite set $\mc{K}$ of \emph{positive} real numbers, also depending only only on $\delta, \rho$, such that if $x\in\mb{S}^{n-1}\setminus\on{Cons}(\delta,\rho)$, then at least one of the following two conclusions is satisfied.
\begin{enumerate}
\item There exist $\kappa, \kappa' \in \mc{K}$ such that 
\[|x_i| \le \frac{\kappa}{\sqrt{n}} \text{ for at least } \nu n \text{ indices }i\in [n], \text{ and }\]
\[\frac{\kappa + \nu'}{\sqrt{n}} < |x_i| \le \frac{\kappa'}{\sqrt{n}} \text{ for at least } \nu n \text{ indices }i\in [n].\]
\item There exist $\kappa, \kappa' \in \mc{K}$ such that
\[\frac{\kappa}{\sqrt{n}} < x_i < \frac{\kappa'}{\sqrt{n}} \text{ for at least } \nu n \text{ indices }i\in [n], \text{ and }\]
\[-\frac{\kappa'}{\sqrt{n}} < x_i < -\frac{\kappa}{\sqrt{n}} \text{ for at least } \nu n \text{ indices }i\in [n].\]
\end{enumerate}
\end{lemma}

\begin{proof}
Let $I_0 := \{i \in [n]: |x_i| \le 4/\sqrt{\delta n} \}$. Since $\snorm{x}_2 =1$, it follows that $|I_0| \ge (1-\delta/16)n$. We consider the following cases. 

\textbf{Case I: }$|\{ i\in I_0: |x_i| < \rho/(10\sqrt{n})\}| \ge \delta n/16$. Since $x \notin \on{Cons}(\delta, \rho)$, there are at least $\delta n$ indices $j \in J$ such that $|x_j| \ge \rho/\sqrt{n}$. Moreover, at least $\delta n/8$ of these indices satisfy $|x_j| \le 4/\sqrt{\delta n}$. Hence, in this case, the first conclusion is satisfied for suitable choice of parameters. 

\textbf{Case II: }$|\{i \in I_0: |x_i| < \rho/(10\sqrt{n})\}| < \delta n/16$ and $|\{i \in I_0: x_i \le -\rho/(10\sqrt{n})\}| \ge \delta n/16$ and $|\{i \in I_0: x_i \ge \rho/(10\sqrt{n})\}| \ge \delta n/16$. In this case, the second conclusion is clearly satisfied for suitable choice of parameters. 

\textbf{Case III: } Either $|\{i \in I_0: x_i \ge \rho/(10\sqrt{n})\}| \ge (1-\delta/4)n$ or $|\{i \in I_0: x_i \le  -\rho/(10\sqrt{n})\}| \ge (1-\delta/4)n$. We assume that we are in the first sub-case; the argument for the second sub-case is similar. We decompose   
\begin{align*}
    [0, 4/\sqrt{\delta n}] = \cup_{\ell=1}^{L} J_\ell,
\end{align*}
where $J_\ell := [(\ell-1)\cdot \rho/(10\sqrt{n}), \ell \cdot \rho/(10\sqrt{n}))$, and $L = O_{\delta, \rho}(1)$. Let 
\begin{align*}
    \ell_0 := \min\{\ell \in L: |\{i \in [n]: x_i \in J_1\cup \dots \cup J_\ell\}| \ge \delta n/16\}.
\end{align*}
Since $x\notin \on{Cons}(\delta, \rho)$, there are at least $\delta n$ indices $j \in [n]$ such that $|x_j - \ell_0\cdot \rho/(10\sqrt{n})| > \rho/\sqrt{n}$. On the other hand, by the assumption of this case and the definition of $\ell_0$, there exist at least $\delta n - (\delta n/4) - (\delta n/16) - (\delta n/16) > \delta n/2$ indices $j \in I_0$ for which $x_j > (\ell_0 + 2)\cdot \rho/(10\sqrt{n})$. Thus, the first conclusion is satisfied for suitable choice of parameters.
\end{proof}

We also isolate the following set of almost-elementary vectors.  
\begin{definition}[Almost-elementary vectors]
For $\delta > 0$ and $i \in [n]$, let 
\[\on{Coord}_i(\delta) := \mb{S}^{n-1}\cap\mb{B}_2^n(e_i,\delta) = \{x\in\mb{S}^{n-1}: \snorm{x-e_i}_2\le\delta\}.\]
We define the set of $\delta$-almost-elementary vectors by
\[\on{Coord}(\delta) := \bigcup_{i=1}^n\on{Coord}_i(\delta).\]
\end{definition}

To prove \cref{prop:compressible}, we will handle almost-elementary vectors and non-almost-elementary vectors separately. We begin with the easier case of non-almost-elementary vectors. 

\subsection{Invertibility on non-almost-elementary vectors}\label{sub:non-almost-elementary}
We begin with a by now standard anticoncentration result due to Kolmogorov-L\'evy-Rogozin, which will also be important later.
\begin{lemma}[\cite{Rog61}]
\label{lem:LKR}
Let $\xi_1,\dots, \xi_n$ be independent random variables. Then, for any real numbers $r_1,\dots,r_n > 0$ and any real $r \ge \max_{i \in [n]}r_i$, we have
\begin{align*}
    \mc{L}\bigg(\sum_{i=1}^{n}\xi_i, r\bigg) \le \frac{C_{\ref{lem:LKR}}r}{\sqrt{\sum_{i=1}^{n}(1-\mc{L}(\xi_i, r_i))r_i^{2}}},
\end{align*}
where $C_{\ref{lem:LKR}}>0$ is an absolute constant. 
\end{lemma}

Using this, we can prove the following elementary fact about sums of independent $\on{Ber}(p)$ random variables.

\begin{lemma}\label{lem:p-characterization}
Fix $p, \delta\in (0,1/2)$. There exists $\theta = \theta(\delta,p) > 0$ such that for all $x\in\mb{S}^{n-1} \setminus \on{Coord}(\delta)$,
\[\mc{L}(b_1x_1+\cdots+b_nx_n,\theta)\le 1-p-\theta,\]
where $b = (b_1,\dots,b_n)$ is a random vector whose coordinates are independent $\on{Ber}(p)$ random variables.
\end{lemma}
\begin{proof}
Since $\on{Coord}(\delta)$ is increasing with $\delta$, it suffices to prove the statement for all sufficiently small $\delta$ (depending on $p$). 
We may assume that $|x_1| \ge |x_2| \ge \cdots \ge |x_n|$. The desired conclusion follows by combining the following two cases.

\textbf{Case 1: } Suppose $|x_2|>\delta^4$. 
We claim that there is some $\theta = \theta(\delta, p)$ for which 
\[\mc{L}(b_1x_1+\dots + b_nx_n, \theta) \le \mc{L}(b_1x_1+b_2x_2,\theta)\le 1-p-\theta.\]
We borrow elements from \cite[Proposition~3.11]{LT20}. The first inequality is trivial. For the second inequality, we note that the random variable $b_1x_1 + b_2x_2$ is supported on the four points $\{0,x_1,x_2,x_1+x_2\}$. Moreover, the sets $\{0,x_1+x_2\}$ and $\{x_1,x_2\}$ are $\delta^4$-separated, and each of these two sets is attained with probability at most $\max\{p^2+(1-p)^2, 2p(1-p)\}< 1-p-\theta$, where the final inequality uses $p < 1/2$. 

\textbf{Case 2: }$|x_2| \le \delta^{4}$. Note that we must have $\snorm{(x_2,\dots,x_n)}_{2} \ge \delta/2$, since otherwise, we would have $\snorm{x - e_1}_{2} < \delta$, contradicting $x \notin \on{Coord}(\delta)$. We claim that there is some $\theta = \theta(\delta, \rho) > 0$ such that
\begin{align*}
    \mc{L}(b_1x_1+\dots + b_n x_n, \theta) \le \mc{L}(b_2 x_2 + \dots + b_n x_n, \theta) \le 1-p-\theta.
\end{align*}
Once again, the first inequality is trivial. The second inequality follows from \cref{lem:LKR} applied with $\xi_i = b_i x_i$ for $i = 2,\dots, n$, $r_i = |x_i|/4$, and $r = \delta^4$, and from our assumption that $\delta$ was small enough in terms of $p$.
\end{proof}

By combining the preceding statement with a standard net argument exploiting the low metric entropy of almost-constant vectors and the well-controlled operator norm of random matrices with independent centered subgaussian entries (cf.~{\cite[Lemma~4.4.5]{Ver18}}), we obtain invertibility on non-almost-elementary vectors. 

\begin{proposition}\label{prop:non-coord}
Fix $p\in(0,1/2)$. Then, for any $\delta' > 0$, there exist $\delta,\rho,\epsilon', n_0 > 0$ (depending on $p,\delta'$) such that for all $n \ge n_0$, 
\[\mb{P}[\exists x\in\on{Cons}(\delta,\rho)\setminus\on{Coord}(\delta'): \snorm{B_n(p)x}_2\le\epsilon'\sqrt{n}]\le (1-p-\epsilon')^{n}.\]
\end{proposition}
\begin{proof}
The argument is closely related to the proof of \cite[Proposition~3.6]{Tik20}, and we omit the standard details. The point is that we can first choose $\epsilon'$ depending on $\delta', p$, then choose $\rho$ sufficiently small depending on $\delta', p, \epsilon'$, and finally, choose $\delta$ sufficiently small depending on all prior choices.
\end{proof}

\subsection{Invertibility on almost-elementary vectors}\label{sub:almost-elementary}
We now prove the much more delicate claim that $\on{Coord}(\delta')$ contributes the appropriate size to the singularity probability.

\begin{proposition}\label{prop:coord-i}
Fix $p\in (0,1/2)$. Given $\theta' > 0$, there exist $\delta',\theta, n_0 > 0$ depending on $p$ and $\theta'$ such that for all $n \ge n_0$,
\[\mb{P}[\exists x\in\on{Coord}(\delta'): \snorm{B_n(p)x}_2\le\exp(-\theta' n)]\le n\cdot \bigg((1-p)^n+(1-p-\theta)^n\bigg).\]
\end{proposition}

Before proceeding to the proof, we need the following preliminary fact, which essentially follows from the seminal work of Rudelson and Vershynin \cite{RV08} (although we were not able to locate the precise statement needed here in the literature).
\begin{lemma}\label{lem:block}
Fix $p \in (0,1)$. For any $c > 0$, there exist $c', n_0 > 0$ depending on $c$ and $p$ for which the following holds. For all $n \ge n_0$ and for any $v \in \mb{R}^{n}$ with $\snorm{v}_2\ge 1$, we have
\[\mb{P}[\exists x\in\mb{R}^{n-1}: \snorm{Ax-v}_2\le 2^{-cn}]\le 2^{-c'n},\]
where $A$ is a random $n\times(n-1)$ matrix with independent $\on{Ber}(p)$ entries.
\end{lemma}
\begin{proof}By reindexing the coordinates, we may write
\[
A = \begin{bmatrix}R\\B\end{bmatrix},\quad v = \begin{bmatrix}v_1\\v'\end{bmatrix},
\]
where $B$ is an $(n-1) \times (n-1)$ matrix, $v' \in \mb{R}^{n-1}$ and $\snorm{v'}_2 \ge 1/2$. Let $\mc{E} = \{s_{n-1}(B)\le 2^{-cn/2}\}$. Then, by an extension of the main result of Rudelson and Vershynin \cite{RV08} (see, e.g., \cite[Theorem~1.3]{JSS20smooth} for a concrete reference), $\mb{P}[\mc{E}]\le 2^{-c_1n}$ for some $c_1 > 0$ depending on $c$ and $p$.  

Moreover, on the event $\mc{E}^{c}$, if there exists some $x \in \mb{R}^{n-1}$ such that $\snorm{Ax - v}_2 \le 2^{-cn}$, then 
\[\snorm{Bx-v'}_2\le 2^{-cn}\implies\snorm{x-B^{-1}v'}_2\le 2^{-cn/2}, \quad \text{ and }\]
\[|Rx-v_1|\le 2^{-cn}\implies |R(B^{-1}v')-v_1|\le 2^{-cn}+n2^{-cn/2}.\]

Let $x_0 := B^{-1}v'/\snorm{B^{-1}v'}_2$; this is a random vector depending on $B$. It follows from a straightforward modification of the argument of Rudelson and Vershynin that there exists a constant $c_2 > 0$, depending on $c$ and $p$, such that with probability at least $1-2^{-c_2n}$, $\mc{L}(\sum_{i=1}^{n-1} b_i \cdot (x_0)_i, 2^{-cn/4}) \le c_2^{-1} 2^{-c_2n}$, where $b_1,\dots, b_n$ are independent $\on{Ber}(p)$ random variables (again, one has to take into account that $\on{Ber}(p)$ does not have mean $0$, but this is not an issue, see, e.g., \cite{JSS20smooth}). Let $\mc{G}$ denote the event (depending on $B$) that this occurs.

Finally, note that since $\snorm{B} \le n$ deterministically, we have $\snorm{B^{-1}v'}_{2} \ge 1/(2n)$. Hence, we see that for all $n$ sufficiently large depending on $c$ and $p$, we have
\begin{align*}
    \mb{P}[\mc{E}^{c} \wedge \snorm{Ax - v}_{2} \le 2^{-cn}]
    &\le \mb{P}[\mc{E}^{c} \wedge |R(B^{-1}v') - v_1|\le 2n\cdot 2^{-cn/2}]\\
    &\le  \mb{P}[\mc{E}^{c} \wedge \mc{G} \wedge |R(B^{-1}v') - v_1|\le 2n\cdot 2^{-cn/2}] + 2^{-c_2n}\\
    &\le \mc{L}\bigg(\sum_{i=1}^{n-1}b_i \cdot (x_0)_i, 4n^2\cdot 2^{-cn/2}\bigg) + 2^{-c_2n}\le 2^{-c_2n/2}, 
\end{align*}
which completes the proof.
\end{proof}

Now we are ready to prove \cref{prop:coord-i}.

\begin{proof}[Proof of \cref{prop:coord-i}]
By taking a union bound, and exploiting the permutation invariance of the distribution of the matrix, it suffices to prove the statement for $\on{Coord}_{1}(\delta')$ and without the additional factor of $n$ on the right hand side. Let $\mc{G}_K$ denote the event that $\snorm{B_n(p) - pJ_{n\times n}} \le K\sqrt{n}$, where $J_{n\times n}$ denotes the $n\times n$ all ones matrix. We fix a choice of $K$ such that $\mb{P}[\mc{G}_K^c] \le \exp(-2n)$, which is possible by \cite[Lemma~4.4.5]{Ver18}.  

For $\delta' \in (0,1/4)$, let $\mc{E}_{\delta'}$ denote the event that there exists some $x \in \on{Coord}_{1}(\delta')$ such that $\snorm{B_n(p)x}_{2} \le \exp(-\theta'n)$. By rescaling, we see that on the event $\mc{E}_{\delta'}$, there exists $y = e_1 + u \in \mb{R}^{n}$ with $u_1 = 0$ and $\snorm{u}_{2} \le 4\delta'$ for which $\snorm{B_n(p)y}_{2} \le 2\exp(-\theta'n)$. For convenience, let $u' := (u_2,\dots,u_n) \in \mb{R}^{n-1}$. Writing 
\[
B_n(p) = \begin{bmatrix}
b_{11}&R\\
S&B_{n-1}
\end{bmatrix},
\]
we see that
\[|Ru'+b_{11}|\le 2\exp(-\theta'n),\qquad\snorm{B_{n-1}u'+S}_2\le 2\exp(-\theta'n).\]
Let $B^{(1)}$ denote the first column of $B_n(p)$ and let $B^{(-1)}$ denote the $n\times (n-1)$ matrix formed by excluding the first column of $B_n(p)$. Then, on the event $\mc{E}_{\delta'} \wedge \mc{G}_K$, we have 
\begin{align*}
    \snorm{B^{(1)} + pJ_{n\times n-1}u'}_{2} 
    &= \snorm{B^{(1)} + B^{(-1)}u' - B^{(-1)}u' + pJ_{n\times n-1}u'}_{2}\\
    &\le \snorm{B^{(1)} + B^{(-1)}u'}_{2} + \snorm{(B^{(-1)} - pJ_{n\times n-1})u'}_{2}\\
    &\le 4\exp(-\theta'n) + K\sqrt{n}\cdot 4\delta'\\
    &\le 8K\delta'\sqrt{n}.
\end{align*}
The key point is the following. Let $\mc{C} := \{x \in \{0,1\}^{n}: \exists \lambda \in \mb{R} \text{ such that } \snorm{x - \lambda 1_n}_{2} \le 8K\delta'\sqrt{n}\}$, where $1_n$ denotes the $n$-dimensional all ones vector. Then, it is readily seen that for any $\epsilon > 0$, there exists $\delta' > 0$ sufficiently small so that
\[\mb{P}[B^{(1)}\in\mc{C}] \le (1-p+\epsilon)^{n}.\]
To summarize, we have shown that for any $\theta', \epsilon > 0$, there exists $\delta' \in (0,1/4)$ such that 
\begin{align}
\label{eq:union-bound-a}
    \mb{P}[\mc{E}_{\delta'}]
    &\le \mb{P}[\mc{E}_{\delta'} \wedge \mc{G}_K] + \exp(-2n) \nonumber \\
    &\le \sum_{a\in\mc{C}}\mb{P}[B^{(1)} = a]\cdot \mb{P}[\exists u' \in \mb{R}^{n-1}, \snorm{u'}_{2} \le 4\delta': \snorm{B^{(-1)}u' + a} \le 2\exp(-\theta'n)].
\end{align}
If we only wanted a bound of the form $(1-p+\epsilon)^n$ on the right hand side, then we would be done. However, since we want a more precise bound, we need to perform a more refined analysis based on whether or not $a = 0$. 

\textbf{Case I: }$a = 0$. The contribution of this term to the sum in \cref{eq:union-bound-a} is exactly $(1-p)^{n}$. 

\textbf{Case II: }$a \neq 0$. Since $a \in \{0,1\}^{n}$, we have $\snorm{a}_{2} \ge 1$. In this case, we can apply \cref{lem:block} with $c = \theta'$ to see that, for all $n$ sufficiently large, 
\[\mb{P}[\exists u' \in \mb{R}^{n-1}, \snorm{u'}_{2} \le 4\delta': \snorm{B^{(-1)}u' + a} \le 2\exp(-\theta'n)] \le 2^{-c'n},\]
where $c' > 0$ depends only on $\theta'$ and $p$. Thus, we see that the contribution of $a\in\mc{C}$, $a\neq 0$ to the sum in \cref{eq:union-bound-a} is at most
\[(1-p+\epsilon)^{n}\times 2^{-c'n}.\]
Since $c'$ does not depend on $\delta'$, we can fix $c'$ (depending on $\theta'$ and $p$), and then choose $\delta'$ sufficiently small so that $\epsilon > 0$ is small enough to make the above product at most $(1-p-\theta)^{n}$, for some $\theta > 0$ depending on the previous parameters.
\end{proof}

We now put everything together to prove \cref{prop:compressible}.

\begin{proof}[Proof of \cref{prop:compressible}]
Let $c > 0$ be as in the statement of the proposition, and choose $\theta' > 0$ such that $\exp(-\theta'n) = 2^{-cn}$. Then, applying  \cref{prop:coord-i} with this choice of $\theta'$, we find $\delta', \theta, n_0$ such that for all $n\ge n_0$, we have the desired estimate for $x \in \on{Coord}(\delta')$ (provided that $\epsilon$ is chosen small enough). Then, we apply  \cref{prop:non-coord} with this choice of $\delta'$ to find $\delta,\rho, n_0 > 0$ such that for all $n \ge n_0$, we have the desired estimate for $x\in \on{Cons}(\delta,\rho)\setminus\on{Coord}(\delta')$, provided again that we choose $\epsilon > 0$ sufficiently small. 
\end{proof}

\section{The conditional threshold function of random lattice points}\label{sec:inversion-of-randomness}
In this section, we establish the key `inversion of randomness' estimate on vectors with poor anticoncentration on the slice of the Boolean hypercube. This follows a direction introduced by Tikhomirov \cite{Tik20}. In this approach, the relevant L\'evy concentration function of a random vector is replaced with certain random averages of functions. One then shows that the random vectors with large values of the L\'evy concentration function are \emph{super-exponentially} rare, by first demonstrating a weaker notion of anticoncentration after revealing $(1-\epsilon)n$ coordinates of the random vector, and then iterating a smoothing procedure on linear-sized pieces of the vector which allows one to bootstrap the strength of anticoncentration considered. Our major challenges lie in the non-independence of the coordinates of a vector on the Boolean slice, as the arguments in \cite{Tik20} rely strongly on the independence structure of the considered model.

\label{sec:random-averaging}
\subsection{Statement and preliminaries}
Let $N,n\ge 1$ be integers and let $0 < \delta < 1/4$, $K_3 > K_2 > K_1 > 1$ be real parameters. We say that $\mc{A}\subseteq\mb{Z}^n$ is $(N,n,K_1,K_2,K_3,\delta)$-admissible if
\begin{itemize}
    \item $\mc{A} = A_1\times\cdots\times A_n$, where each $A_i$ is a subset of $\mb{Z}$,
    \item $|A_1|\cdots|A_n|\le (K_3N)^n$,
    \item $\max_i\max\{|a|: a\in A_i\}\le nN$,
    \item $A_i$ is an integer interval of size at least $2N+1$ for $i > 2\delta n$, and either (P1) and (P2) or (Q1) and (Q2) hold:
\end{itemize}
\begin{enumerate}[(P1)]
    \item $A_{2i}$ is an integer interval of size at least $2N+1$ and is contained in $[-K_1N,K_1N]$ for $i\le\delta n$,
    \item $A_{2i-1}$ is symmetric about $0$ and is a union of two integer intervals of total size at least $2N$ with $A_{2i-1}\cap[-K_2N,K_2N] = \emptyset$ for $i\le\delta n$.
\end{enumerate}
\begin{enumerate}[(Q1)]
    \item $A_{2i}$ is an integer interval of size at least $2N+1$ and is contained in $[K_1N,K_2N]$ for $i\le\delta n$,
    \item $A_{2i-1}$ is an integer interval of size at least $2N+1$ contained in $[-K_2N,-K_1N]$ for $i\le\delta n$.
\end{enumerate}
\begin{remark}
The conditions on sets up to $2\delta n$, upon rescaling (and application of one of an exponential-sized set of permutations), correspond to the fact that the potential normal vectors considered are not $(\delta,\rho)$-almost constant, with (P1), (P2) corresponding to the first conclusion of \cref{lem:nonconstant-admissible} and (Q1), (Q2) corresponding to the second conclusion.
\end{remark}

Let $\mc{A} = A_1\times\cdots\times A_n$ be an $(N,n,K_1,K_2,K_3,\delta)$-admissible set, and let $(X_1,\ldots,X_n)$ be the random vector uniformly distributed on $\mc{A}$. For any $f: \mb{Z}\to\mb{R}$ and any $0\le s\le\ell\le n$, define the random function (depending on the randomness of $X_1,\dots,X_n$):
\[f_{\mc{A},s,\ell}(t) := \mb{E}_b\bigg[f\bigg(t+\sum_{j=1}^\ell b_jX_j\bigg)\bigg|\sum_{j=1}^\ell b_j=s\bigg] = \frac{1}{\binom{\ell}{s}}\sum_{v\in\{0,1\}_s^\ell}f(t+v_1X_1+\cdots+v_\ell X_\ell),\]
where $\mb{E}_b$ denotes the expectation over a uniform random vector $b = (b_1,\dots,b_n)\in\{0,1\}^n$.

The main result of this section is the following. 
\begin{theorem}\label{thm:inversion-of-randomness}
For $0 < \delta < 1/4$, $K_3 > K_2 > K_1 > 1$, $p\in(0,1/2]$, $\epsilon\ll p$, and a given parameter $M\ge 1$, there are  $L_{\ref{thm:inversion-of-randomness}} = L_{\ref{thm:inversion-of-randomness}}(p,\epsilon,\delta,K_1,K_2,K_3) > 0$ and $\gamma_{\ref{thm:inversion-of-randomness}} = \gamma_{\ref{thm:inversion-of-randomness}}(p,\epsilon,\delta,K_1,K_2,K_3)\in(0,\epsilon)$ independent of $M$, and $n_{\ref{thm:inversion-of-randomness}} = n_{\ref{thm:inversion-of-randomness}}(p,\epsilon,\delta,K_1,K_2,K_3,M)\ge 1$ and $\eta_{\ref{thm:inversion-of-randomness}} = \eta_{\ref{thm:inversion-of-randomness}}(p,\epsilon,\delta,K_1,K_2,K_3,M)\in(0,1]$ such that the following holds. 

Let $n\ge n_{\ref{thm:inversion-of-randomness}}$ and $1\le N\le\binom{n}{pn}\exp(-\epsilon n)$. Let $\mc{A}$ be $(N,n,K_1,K_2,K_3,\delta)$-admissible, and consider a nonnegative function $f\in\ell_1(\mb{Z})$ with $\snorm{f}_1 = 1$ such that $\log_{2} f$ is $\eta_{\ref{thm:inversion-of-randomness}}$-Lipschitz. Then, for any $m\in[pn-\gamma_{\ref{thm:inversion-of-randomness}}n,pn+\gamma_{\ref{thm:inversion-of-randomness}}n]$,
\[\mb{P}[\snorm{f_{\mc{A},m,n}}_\infty\ge L_{\ref{thm:inversion-of-randomness}}(N\sqrt{n})^{-1}]\le\exp(-Mn).\]
\end{theorem}

\begin{remark}
The key point, as in \cite[Theorem~4.2]{Tik20}, is that $L_{\ref{thm:inversion-of-randomness}}$ does not depend on $M$. Moreover, as in  \cite{Tik20}, one can easily see that the log-Lipschitz condition is necessary by considering the case when $f = \mbm{1}_{\{0\}}$, and analyzing the event where $X_i \in q\cdot \mb{Z}$ for all $i \in [n]$.
\end{remark}

In this paper, we will use the above theorem through the following immediate corollary.

\begin{definition}
\label{def:conditional-levy}
Let $p \in (0,1/2]$, $\gamma \in (0,p)$, and $r\ge 0$. For a vector $(x_1,\dots,x_n) \in \mb{R}^{n}$, we define
\begin{align*}
    \mc{L}_{p,\gamma}\bigg(\sum_{i=1}^nb_i x_i, r\bigg) := \sup_{z\in \mb{R}}\mb{P}\bigg[\bigg|\sum_{i=1}^nb_i x_i - z \bigg| \le r \bigg| \sum_{i=1}^{n} b_i \in [pn - \gamma n, pn + \gamma n]\bigg],
\end{align*}
where $b_1,\dots,b_n$ are independent $\on{Ber}(p)$ random variables. 
\end{definition}

\begin{corollary}\label{cor:threshold-inversion}
For $0 < \delta < 1/4$, $K_3 > K_2 > K_1 > 1$, $p\in(0,1/2]$, $\epsilon\ll p$, and a given parameter $M\ge 1$, there are  $L_{\ref{cor:threshold-inversion}} = L_{\ref{cor:threshold-inversion}}(p,\epsilon,\delta,K_1,K_2,K_3) > 0$ and $\gamma_{\ref{cor:threshold-inversion}} = \gamma_{\ref{cor:threshold-inversion}}(p,\epsilon,\delta,K_1,K_2,K_3)\in(0,\epsilon)$ independent of $M$ and $n_{\ref{cor:threshold-inversion}} = n_{\ref{cor:threshold-inversion}}(p,\epsilon,\delta,K_1,K_2,K_3,M)\ge 1$ such that the following holds. 

Let $n\ge n_{\ref{cor:threshold-inversion}}$, $1\le N\le\binom{n}{pn}\exp(-\epsilon n)$, and $\mc{A}$ be $(N,n,K_1,K_2,K_3,\delta)$-admissible. Suppose also that  $\gamma\le\gamma_{\ref{cor:threshold-inversion}}$. Then 
\[\bigg| \bigg\{x\in \mc{A}: \mc{L}_{p,\gamma}\bigg(\sum_{i=1}^nb_ix_i,\sqrt{n}\bigg)\ge L_{\ref{cor:threshold-inversion}}N^{-1}\bigg\}\bigg|\le e^{-Mn}|\mc{A}|.\]
\end{corollary}
\begin{remark}
In our application, $K_1,K_2, K_3, \delta$ will be determined by the parameters coming from \cref{prop:compressible}, and $\epsilon$ depends on the gap between $p$ and $1/2$. The parameter $M \ge 1$ will be chosen to be sufficiently large in order to overcome an exponential-sized union bound, again depending on the parameters coming from \cref{prop:compressible}.
\end{remark}
\begin{proof}
This is identical to the deduction of \cite[Corollary~4.3]{Tik20} by applying \cref{thm:inversion-of-randomness} to $f(t):= 2^{-|t|/\sqrt{n}}/\iota$, where $t \in \mb{Z}$ and $\iota$ is a normalizing factor. We apply this deduction separately for all $m\in[pn-\gamma_{\ref{thm:inversion-of-randomness}}n,pn+\gamma_{\ref{thm:inversion-of-randomness}}n]$, and then conclude using a union bound.
\end{proof}

The proof of \cref{thm:inversion-of-randomness} requires a couple of preliminary anticoncentration statements for the slice, which we record here.

\begin{lemma}\label{lem:slkr}
Let $\sigma, \lambda \in (0,1/3)$ and $r > 0$. Let $A = \{a_1,\ldots,a_n\}$ be a set of real numbers for which there exist disjoint subsets $A_1, A_2 \subseteq A$ such that $|A_1|, |A_2| \ge \sigma n$ and such that $|a_i - a_j| \ge r$ for all $a_i \in A_1, a_j \in A_2$. Then, there exists $C_{\ref{lem:slkr}} = C_{\ref{lem:slkr}}(\lambda, \sigma)$ such that for all $m\in[\lambda n,(1-\lambda)n]$ we have
\[\mc{L}\bigg(\sum_{i=1}^{n}a_ib_i, r\bigg) \le\frac{C_{\ref{lem:slkr}}}{\sqrt{n}},\]
where $(b_1,\dots, b_n)$ is a random vector uniformly distributed on $\{0,1\}^{n}_m$.
\end{lemma}
\begin{proof}By reindexing the coordinates of $A$, we may assume that for $i \in [\sigma n]$, $a_{2i - 1} \in A_1$ and $a_{2i} \in A_2$. In particular, for $i \in [\sigma n]$, we have $|a_{2i} - a_{2i - 1}| \ge r$.  Note that $\sum_{i=1}^{n} a_i b_i$
has the same distribution as
\[\sum_{i>2\sigma n}a_i b_i + \sum_{j\le \sigma n}\bigg(a_{2j-1}b_{2j-1} + a_{2j}b_{2j} + b_j'(b_{2j} - b_{2j-1})(a_{2j-1} - a_{2j})\bigg),\]
where $b_1',\dots,b_{\sigma n}'$ are i.i.d.~$\on{Ber}(1/2)$ random variables. Next, note that by a standard large deviation estimate, we have
\begin{align}
\label{eq:rerandomize}
\mb{P}[|\{j \in [\sigma n]: b_{2j-1} + b_{2j} = 1\}| \le c(\sigma, \lambda)n] \le \exp(-c(\sigma, \lambda)n),
\end{align}
where $c(\sigma, \lambda) > 0$ is a constant depending only on $\sigma$ and $\lambda$. On the other hand, on the complement of this event, we may conclude by applying \cref{lem:LKR} to \cref{eq:rerandomize}, using only the randomness in $b_1',\dots, b_{\sigma n}'$.
\end{proof}

Using the previous lemma along with the definition of an admissible set allows us to obtain the following. The proof is essentially identical to \cite[Lemma~4.4]{Tik20} and we omit details. 
\begin{lemma}\label{lem:slice-levy-kolmogorov-rogozin}
Fix $\lambda\in(0,1/3)$ and $\delta_0\in(0,1/4)$. Let $\mc{A}$ be $(N,n,K_1,K_2,K_3,\delta)$-admissible for some integer parameters $N,n$ and real parameters $\delta\in[\delta_0,1/4)$, $K_3 > K_2 > K_1 > 1$. Suppose that $n > n_{\ref{lem:slice-levy-kolmogorov-rogozin}}(\lambda,\delta_0,K_1,K_2,K_3)$, $\ell \ge \delta_0n$, and $s \in [\lambda \cdot \ell , (1-\lambda)\cdot\ell]$. Then, for any integer interval $J$ of length at least $N$,
\[\sum_{t\in J}f_{\mc{A},s,\ell}(t)\le\frac{C_{\ref{lem:slice-levy-kolmogorov-rogozin}}(\lambda,\delta_0,K_1,K_2)|J|}{N\sqrt{n}}.\]
\end{lemma}

\subsection{Preprocessing the function}

As a critical first step, we prove a version of \cref{thm:inversion-of-randomness} in which $L_{\ref{thm:inversion-of-randomness}}$ depends on $M$. Later, using \cref{prop:slice-linfty-decrement}, we will show how to rectify this defect. Note that for the proposition below, we do not require the log-Lipschitz assumption. This initial estimate plays an analogous role to \cite[Proposition~4.5]{Tik20}, but our situation is more delicate due to the dependencies induced by working on the slice, and the need to extract terms equivalent to the entropy of the slice.
\begin{proposition}\label{prop:rough-Linfty}
For $0 < \delta < 1/4$, $K_3 > K_2 > K_1 > 1$, $p\in(0,1/2]$, $\epsilon\ll p$, and a given parameter $M\ge 1$, there is $\gamma_{\ref{prop:rough-Linfty}} = \gamma_{\ref{prop:rough-Linfty}}(p,\epsilon,\delta,K_1,K_2,K_3)\in(0,\epsilon)$ independent of $M$ and there are $L_{\ref{prop:rough-Linfty}} = L_{\ref{prop:rough-Linfty}}(p,\epsilon,\delta,K_1,K_2,K_3,M) > 0$ and $n_{\ref{prop:rough-Linfty}} = n_{\ref{prop:rough-Linfty}}(p,\epsilon,\delta,K_1,K_2,K_3,M)\ge 1$ such that the following holds. 

Let $n\ge n_{\ref{prop:rough-Linfty}}$, $1\le N\le\binom{n}{pn}\exp(-\epsilon n)$, and $\mc{A}$ be $(N,n,K_1,K_2,K_3,\delta)$-admissible. Let $f$ be a nonnegative function in $\ell_1(\mb{Z})$ with $\snorm{f}_1 = 1$. Then, for all $s\in[p\ell-\gamma_{\ref{prop:rough-Linfty}}\ell,p\ell+\gamma_{\ref{prop:rough-Linfty}}\ell]$ and $\ell\in[(1-\gamma_{\ref{prop:rough-Linfty}})n,n]$, we have
\[\mb{P}\bigg[\snorm{f_{\mc{A},s,\ell}}_\infty\ge L_{\ref{prop:rough-Linfty}}(N\sqrt{n})^{-1}\bigg]\le\exp(-Mn).\]
\end{proposition}

The starting point of the proof is the following trivial recursive relation:
\[f_{\mc{A},s,\ell}(t) = \bigg(1-\frac{s}{\ell}\bigg)f_{\mc{A},s,\ell-1}(t) + \frac{s}{\ell}f_{\mc{A},s-1,\ell-1}(t+X_\ell)\]
for all $1\le\ell\le n$ and $0\le s\le\ell$, noting that if $s\in\{0,\ell\}$, the one undefined term will have a coefficient of $0$. Note also that $f_{\mc{A},0,\ell} = f$.

\begin{definition}[Step record and averaging sequence]
\label{def:record-averaging}
Fix $f, \mc{A}, s, \ell$, a point $t \in \mb{Z}$, and a choice of $X = (X_1,\dots,X_n)$. For such a choice, we define the \emph{averaging sequence} $(t_i)_{i=0}^{\ell}$ and \emph{step record} $(w_i)_{i=1}^{\ell}$ as follows:
\begin{itemize}
    \item $t_\ell := t$,
    \item Since \[h_\ell := f_{\mc{A},s,\ell}(t_\ell) = \bigg(1-\frac{s}{\ell}\bigg)f_{\mc{A},s,\ell-1}(t_\ell) + \frac{s}{\ell}f_{\mc{A},s-1,\ell-1}(t_\ell+X_\ell),\]
    at least one of the two terms $f_{\mc{A}, s, \ell - 1}(t_\ell)$ and $f_{\mc{A}, s-1, \ell-1}(t_\ell + X_\ell)$ has a positive coefficient and is at least $h_\ell$. If it is the former, set $w_\ell = 0$, else set $w_\ell = 1$. 
    \item Set $t_{\ell - 1} := t_\ell + w_\ell X_\ell$, $h_{\ell-1} := f_{\mc{A}, s-w_\ell, \ell-1}(t_{\ell-1})$, and repeat with $t_{\ell-1},s-w_\ell,\ell-1$.
\end{itemize}
It will be convenient to denote
\begin{itemize}
    \item $W_i:= \sum_{j=1}^{i}w_j$ and $\ol{W}_i := W_i/i$ for all $i \in [\ell]$.  
\end{itemize}
\end{definition}

We record some immediate consequences of the above definition. 
\begin{itemize}
\item $W_\ell = s$.
\item $t_{i-1} = t_i + w_i X_i$ for all $i \in [\ell]$.
\item $f_{\mc{A}, W_i, i}(t_i) = (1-\ol{W}_i)f_{\mc{A}, W_i, i-1}(t_i) + \ol{W}_if_{\mc{A}, W_i - 1, i-1}(t_i +X_i)$.
\item $h_i = f_{\mc{A}, W_i, i}(t_i)$.
\item $f(t_0) = h_0 \ge h_1 \ge \cdots \ge h_{\ell} = f_{\mc{A}, s, \ell}(t)$.
\end{itemize}
\begin{definition}[Drops and robust steps]
\label{def:drop-robust}
With notation as above, and for $i \in [\ell]$:
\begin{itemize}
\item For $\lambda \in (0,1)$, we say that step $i$ is \emph{$\lambda$-robust} if 
\begin{align*}
    \ol{W}_i \in (\lambda, 1-\lambda).
\end{align*}
\item For $R > 0$, we say that there is an \emph{$R$-drop} at step $i$ if
\begin{align*}
    f_{\mc{A}, W_i - y, i-1}(t_{i-1} + z X_i) \le \frac{R}{N\sqrt{n}}
\end{align*}
for all $y \in \{0,1\}$ and $z \in \{-1, 1\}$.
\end{itemize}
\end{definition}

The next lemma, which is the key point of this subsection, shows that the event that $\snorm{f_{\mc{A},s,\ell}}_{\infty}$ is `large' is contained in the event that there is an averaging sequence with linearly many \emph{robust} steps which do not see an $R$-drop, and will enable us to prove \cref{prop:rough-Linfty} using a union bound argument.
\begin{lemma}\label{lem:entropy}
Let $\mc{A}, f, N, \epsilon, p$ be as in \cref{prop:rough-Linfty}, and let $L \ge 1$. Then, there exist $\lambda_{\ref{lem:entropy}} = \lambda_{\ref{lem:entropy}}(p,\epsilon) \in (0,1/3)$, $\gamma_{\ref{lem:entropy}} = \gamma_{\ref{lem:entropy}}(p,\epsilon) \in (0,1)$, and $n_{\ref{lem:entropy}} = n_{\ref{lem:entropy}}(p,\epsilon)$ for which the following holds.

Let $n \ge n_{\ref{lem:entropy}}$, $R = \gamma_{\ref{lem:entropy}}L$, and let  $\ell\in[(1-\gamma_{\ref{lem:entropy}})n,n]$ and $s\in[p\ell-\gamma_{\ref{lem:entropy}}\ell,p\ell+\gamma_{\ref{lem:entropy}}\ell]$. Then, for $(X_1,\dots,X_n) \in \mc{A}$,
\[\snorm{f_{\mc{A},s,\ell}}_\infty\ge L(N\sqrt{n})^{-1}\]
implies that there exists some $t\in\mb{Z}$ with $f_{\mc{A},s,\ell}(t) \ge L(N\sqrt{n})^{-1}$ so that its averaging sequence $(t_i)_{i=0}^{\ell}$ and step record $(w_i)_{i=1}^{\ell}$ satisfy
\[\#\{i\in[\ell]: \emph{step }i\emph{ is }\lambda_{\ref{lem:entropy}}\emph{-robust and is not an }R\emph{-drop}\}\ge\gamma_{\ref{lem:entropy}} n.\]
\end{lemma}
\begin{proof}
Fix $(X_1,\dots,X_n) \in \mc{A}$ such that $\snorm{f_{\mc{A}, s, \ell}}_{\infty} \ge L(N\sqrt{n})^{-1}$. Then, there must exist some $t \in \mb{Z}$ such that $f_{\mc{A},s,\ell}(t) \ge L(N\sqrt{n})^{-1}$. We will show that the conclusion of the lemma is satisfied for this $t$, for suitable choice of $\gamma_{\ref{lem:entropy}}, \lambda_{\ref{lem:entropy}}$. Below, we will freely use the notation and relations in \cref{def:drop-robust,def:record-averaging}. Let $(t_i)_{i=0}^{\ell}$ and $(w_i)_{i=1}^{\ell}$ denote, respectively, the averaging sequence and step record of $t$. 
Note that
\[L(N\sqrt{n})^{-1}\le f_{\mc{A},s,\ell}(t) = h_0\prod_{i=1}^\ell\frac{h_i}{h_{i-1}} \le h_{\ell-1} \le \cdots \le h_0.\]

We begin by controlling the ratios $h_i/h_{i-1}$ at steps $i$ which are $R$-drops. Hence, suppose that step $i$ is an $R$-drop. If $w_i = 0$, then $W_i = W_{i-1}$ and $t_i = t_{i-1}$, so that
\begin{align*}
    \frac{h_i}{h_{i-1}} 
    &= (1-\ol{W}_i)\frac{f_{\mc{A}, W_i, i-1}(t_i)}{f_{\mc{A}, W_{i-1}, i-1}(t_{i-1})} + \ol{W}_i \frac{f_{\mc{A}, W_i-1, i-1}(t_i + X_i)}{f_{\mc{A}, W_{i-1}, i-1}(t_{i-1})}\\
    &=  (1-\ol{W}_i)\frac{f_{\mc{A}, W_i, i-1}(t_i)}{f_{\mc{A}, W_i, i-1}(t_i)} + \ol{W}_i \frac{f_{\mc{A}, W_i-1, i-1}(t_{i-1} + X_i)}{h_{i-1}}\\
    &\le (1-\ol{W}_i) + \ol{W}_i \frac{R(N\sqrt{n})^{-1}}{L(N\sqrt{n})^{-1}}\\
    & = 1-\ol{W}_i + \ol{W}_i \gamma_{\ref{lem:entropy}}. 
\end{align*}
Combining this with a similar computation for the case when $w_i = 1$, we see that if step $i$ is an $R$-drop, then
\begin{align}
\label{eq:ratio-bound}
    \frac{h_i}{h_{i-1}} \le \bigg( 1-\ol{W}_i + \ol{W}_i \gamma_{\ref{lem:entropy}}\bigg)^{1-w_i}\bigg( \ol{W}_i + (1-\ol{W}_i) \gamma_{\ref{lem:entropy}}\bigg)^{w_i}.
\end{align}

Note that if step $i$ is $\lambda_{\ref{lem:entropy}}$-robust, each of the terms in the product on the right-hand side is at least $\lambda_{\ref{lem:entropy}}$. Since at least one of the two terms is $1$, it follows that the product is at least $\lambda_{\ref{lem:entropy}}$. Therefore, for any step $i$ which is $\lambda_{\ref{lem:entropy}}$-robust, we have
\begin{align}
\label{eq:robust-step-bound}
    \lambda_{\ref{lem:entropy}}
    \le \bigg( 1-\ol{W}_i + \ol{W}_i \gamma_{\ref{lem:entropy}}\bigg)^{1-w_i}\bigg( \ol{W}_i + (1-\ol{W}_i) \gamma_{\ref{lem:entropy}}\bigg)^{w_i} \le 
    \bigg(1 - \ol{W}_i\bigg)^{1-w_i}\bigg(\ol{W}_i\bigg)^{w_i}\bigg(1+\frac{\gamma_{\ref{lem:entropy}}}{\lambda_{\ref{lem:entropy}}}\bigg),
\end{align}
where the final inequality uses $\max\{\ol{W}_i/(1-\ol{W}_i), (1-\ol{W}_i)/\ol{W}_i\} \le 1/\lambda_{\ref{lem:entropy}}$ at any $\lambda_{\ref{lem:entropy}}$-robust step $i$. 

Now, let $I \subseteq [\ell]$ denote the steps $i$ which are $\lambda_{\ref{lem:entropy}}$-robust, and let $J \subseteq I$ denote the steps $i$ which are \emph{not} $R$-drops (so that $I \setminus J$ is the set of $\lambda_{\ref{lem:entropy}}$-robust $R$-drops). Our goal is to provide a lower bound on $|J|$.  

Since $h_i\ge h_{i-1}$ and $h_0\le\snorm{f}_\infty\le\snorm{f}_1 = 1$, we have
\begin{align}
\label{eq:telescoping-bound}
    L(N\sqrt{n})^{-1} &\le \prod_{i \in I \setminus J}\frac{h_i}{h_{i-1}} \le \prod_{i \in I\setminus J}\bigg( 1-\ol{W}_i + \ol{W}_i \gamma_{\ref{lem:entropy}}\bigg)^{1-w_i}\bigg( \ol{W}_i + (1-\ol{W}_i) \gamma_{\ref{lem:entropy}}\bigg)^{w_i} \nonumber \\
    &= \frac{\prod_{i \in I}\bigg( 1-\ol{W}_i + \ol{W}_i \gamma_{\ref{lem:entropy}}\bigg)^{1-w_i}\bigg( \ol{W}_i + (1-\ol{W}_i) \gamma_{\ref{lem:entropy}}\bigg)^{w_i}}{\prod_{i \in J}\bigg( 1-\ol{W}_i + \ol{W}_i \gamma_{\ref{lem:entropy}}\bigg)^{1-w_i}\bigg( \ol{W}_i + (1-\ol{W}_i) \gamma_{\ref{lem:entropy}}\bigg)^{w_i}} \nonumber \\
    &\le \frac{(1+\gamma_{\ref{lem:entropy}}/\lambda_{\ref{lem:entropy}})^{|I|}\prod_{i\in I}(1-\ol{W}_i)^{1-w_i}\ol{W}_i^{w_i}}{\lambda_{\ref{lem:entropy}}^{|J|}} \nonumber \\
    &= (1+\gamma_{\ref{lem:entropy}}/\lambda_{\ref{lem:entropy}})^{|I|}\lambda_{\ref{lem:entropy}}^{-|J|}\prod_{i\in[\ell]}(1-\ol{W}_i)^{1-w_i}\ol{W}_i^{w_i} \prod_{i\in[\ell]\setminus I}(1-\ol{W}_i)^{-(1-w_i)}\ol{W}_i^{-w_i} \nonumber \\
    & = (1+\gamma_{\ref{lem:entropy}}/\lambda_{\ref{lem:entropy}})^{|I|} \cdot \lambda_{\ref{lem:entropy}}^{-|J|} \cdot \binom{\ell}{s}^{-1} \cdot \prod_{i\in[\ell]\setminus I}(1-\ol{W}_i)^{-(1-w_i)}\ol{W}_i^{-w_i};
\end{align}
here, the first line uses $h_i/h_{i-1} \le 1$ and \cref{eq:ratio-bound}, the third line uses \cref{eq:robust-step-bound}, and the last line uses the identity
\[\prod_{i\in[\ell]}(1-\ol{W}_i)^{1-w_i}\ol{W}_i^{w_i} = \binom{\ell}{s}^{-1},\]
noting that both sides are equal to the probability that a uniformly random sample from $\{0,1\}^{\ell}_{s}$ returns $(w_1,\dots,w_\ell)$.

Note that the first and the third terms in the final product in \cref{eq:telescoping-bound} are easy to suitably control (by taking $\gamma_{\ref{lem:entropy}}$ and $\lambda_{\ref{lem:entropy}}$ to be sufficiently small). As we will see next, these parameters also allow us to make the last term at most $\exp(c\epsilon n)$ for any constant $c > 0$. 

Let $K \subseteq [\ell] \setminus I$ denote those indices $i$ such that either $\ol{W}_i \le \lambda_{\ref{lem:entropy}}$ and $w_i = 0$, or $\ol{W}_i \ge (1-\lambda_{\ref{lem:entropy}})$ and $w_i  = 1$. Then, 
\begin{align}
\label{eq:product-K}
\prod_{i\in K}(1-\ol{W}_i)^{-(1-w_i)}\ol{W}_i^{-w_i} \le (1-\lambda_{\ref{lem:entropy}})^{-|K|}.
\end{align}
It remains to bound 
\[\prod_{i\in [\ell]\setminus (I\cup K)}(1-\ol{W}_i)^{-(1-w_i)}\ol{W}_i^{-w_i}.\]
Note that for every $i \in [\ell]\setminus (I\cup K)$, we either have $\ol{W}_i \le \lambda_{\ref{lem:entropy}}$ and $w_i = 1$, or $\ol{W}_i \ge 1-\lambda_{\ref{lem:entropy}}$ and $w_i = 0$.  
The following is the key point: let $i_1,\dots, i_k \in [\ell]\setminus (I\cup K)$ be all indices satisfying the first condition. Then, for all $j \in [k]$, we have
\begin{align*}
    j = w_{i_1}+\dots+w_{i_j}\le W_{i_j} \le \lambda_{\ref{lem:entropy}}\ell. 
\end{align*}
Hence, 
\begin{align*}
    k \le \lambda_{\ref{lem:entropy}}\ell \quad \text{ and } \quad \ol{W}_{i_j}^{-1} \le i_j/j\le \ell/j.
\end{align*}
Combining this with a similar calculation for indices in $[\ell]\setminus K$ satisfying the second condition, we have
\begin{align}
    \label{eq:K-complement-bound}
 \prod_{i\in [\ell]\setminus (I\cup K)}(1-\ol{W}_i)^{-(1-w_i)}\ol{W}_i^{-w_i} \le \bigg(\prod_{j=1}^{\lceil\lambda_{\ref{lem:entropy}}\ell\rceil}\frac{\ell}{j}\bigg)^2 \le \bigg(\frac{e}{\lambda_{\ref{lem:entropy}}}\bigg)^{4\lambda_{\ref{lem:entropy}}\ell}.   
\end{align}
Substituting \cref{eq:product-K} and \cref{eq:K-complement-bound} in \cref{eq:telescoping-bound}, we have
\begin{align}
\label{eq:J-bound}
    Ln^{-1/2}\binom{n}{pn}^{-1}\exp(\epsilon n) \le \lambda_{\ref{lem:entropy}}^{-|J|}\cdot \bigg(1+\frac{\gamma_{\ref{lem:entropy}}}{\lambda_{\ref{lem:entropy}}}\bigg)^{\ell}\cdot \binom{\ell}{s}^{-1}\cdot (1-\lambda_{\ref{lem:entropy}})^{-\ell}\cdot \bigg(\frac{e}{\lambda_{\ref{lem:entropy}}}\bigg)^{4\lambda_{\ref{lem:entropy}}\ell}.
\end{align}

We will first choose $\lambda_{\ref{lem:entropy}}$, and then choose some  $\gamma_{\ref{lem:entropy}} < \lambda_{\ref{lem:entropy}}^{2}$. Note that, enforcing the constraint $\gamma_{\ref{lem:entropy}} < \lambda_{\ref{lem:entropy}}^{2}$, we can choose $\lambda_{\ref{lem:entropy}}$ sufficiently small depending on $\epsilon$ and $p$ so that the second term, the fourth term, and the fifth term in the product in \cref{eq:J-bound} are each bounded above by $\exp(\epsilon n/10)$ and so that the third term is bounded above by $\binom{n}{pn}^{-1}\exp(\epsilon n/10)$. Hence, we can choose $\lambda_{\ref{lem:entropy}}$ depending on $\epsilon$ and $p$ such that
$$n^{-1/2}\exp(\epsilon n/2) \le \lambda_{\ref{lem:entropy}}^{-|J|}.$$
Now, for all $n$ sufficiently large depending on $\epsilon$, we can find $\gamma_{\ref{lem:entropy}}$ sufficiently small depending on $\epsilon, \lambda_{\ref{lem:entropy}}$ such that 
$|J| \ge \gamma_{\ref{lem:entropy}}n$. This completes the proof.
\end{proof}

We can now prove \cref{prop:rough-Linfty}.
\begin{proof}[Proof of \cref{prop:rough-Linfty}]
We use \cref{lem:entropy} along with a union bound. For controlling individual events in the union, we will use the following. Consider a step record $(w_i)_{i=1}^{\ell}$. Then, if step $i$ is $\lambda_{\ref{lem:entropy}}$-robust with respect to $(w_i)_{i=1}^{\ell}$ and if $i > \delta_0 n$, then for any $t \in \mb{Z}$, $y \in \{0,1\}$ and $z \in \{-1,1\}$, by \cref{lem:slice-levy-kolmogorov-rogozin}, we have 
\begin{align*}
\mb{E}[f_{\mc{A},W_i-y,i-1}(t+zX_i)|X_1,\ldots,X_{i-1}] &= \frac{1}{|A_i|}\sum_{\tau\in t+z A_i}f_{\mc{A},W_i - y,i-1}(\tau)\\
&\le\frac{2C_{\ref{lem:slice-levy-kolmogorov-rogozin}}(\lambda_{\ref{lem:entropy}}/2,\delta_0 ,K_1,K_2)}{N\sqrt{n}}.
\end{align*}
Here, we have used $i-1 \ge \delta_0 n$, $W_i - y \in [\lambda_{\ref{lem:entropy}}(i-1)/2, (1-\lambda_{\ref{lem:entropy}}/2)(i-1)]$ and that $A_i$ is a union of at most $2$ integer intervals of length at least $N$. Now, consider $t \in \mb{Z}$ with averaging sequence $(t_i)_{i=0}^{\ell}$ and step record $(w_i)_{i=1}^{\ell}$. Note that, given the `starting point' $t_0$ of the averaging sequence, the points $t_1,\dots, t_{i-1}$ are determined by $X_1,\dots,X_{i-1}$. In particular, the event that step $i$ is not an $R$-drop is determined by $t_0, X_1,\dots, X_i, w_1,\dots,w_i$. Therefore, by Markov's inequality, we see that for any $\lambda_{\ref{lem:entropy}}$-robust step $i$ with $i > \delta_0n$, given the step record $(w_i)_{i=1}^\ell$ and the starting point $t_0$ of the averaging sequence $(t_i)_{i=0}^\ell$,
\begin{align}
\label{eq:R-drop-bound}
    \mb{P}[ \text{step } i \text{ is not an }R\text{-drop} | X_1,\dots, X_{i-1}] \le \frac{8C_{\ref{lem:slice-levy-kolmogorov-rogozin}}(\lambda_{\ref{lem:entropy}}/2,\delta_0 ,K_1,K_2)}{R}. 
\end{align}
From here on, the proof closely follows the proof of \cite[Proposition~4.5]{Tik20}. Fix parameters as given in the proposition statement. Let $\lambda_{\ref{lem:entropy}} = \lambda_{\ref{lem:entropy}}(p,\epsilon)$. 
We choose $\gamma_{\ref{prop:rough-Linfty}} = \gamma_{\ref{lem:entropy}}(p,\epsilon)$. Further, we set  $R = \gamma_{\ref{lem:entropy}}L$, where $L \ge 1$ will be chosen later.  

Let $\mc{E}_L$ denote the event that  $\snorm{f_{\mc{A},s,\ell}}_\infty\ge L(N\sqrt{n})^{-1}$. For $(X_1,\dots,X_n) \in \mc{E}_L$, by \cref{lem:entropy}, there exists $t \in \mb{Z}$ with $f_{\mc{A},s,\ell}(t) \ge L(N\sqrt{n})^{-1}$ with averaging sequence $(t_i)_{i=0}^{\ell}$ and step record $(w_i)_{i=1}^{\ell}$ such that
\[\#\{i\in[\ell]: \text{step }i\text{ is }\lambda_{\ref{lem:entropy}}\text{-robust and is not an }R\text{-drop}\}\ge\gamma_{\ref{lem:entropy}} n.\]
We call such an averaging sequence $(t_i)_{i=0}^{\ell}$ a \emph{witnessing sequence}. 

Taking a union bound over the choice of the step record is not costly, and note that given $(X_1,\dots,X_n)$ and the step record, the witnessing  sequence is completely determined by its starting point $t_0$. Furthermore, the definition of the witnessing sequence and the definition of $f_{\mc{A},s,\ell}$ easily show that $f(t_0)\ge f_{\mc{A},s,\ell}(t)\ge L(N\sqrt{n})^{-1}$ so
\[t_0\in\{\tau\in\mb{Z}: f(\tau)\ge (N\sqrt{n})^{-1}\} =: \mc{D}.\]
Note that $\mc{D}$ is a deterministic set depending only on $f$. Further, since $\snorm{f}_1 = 1$ we see that
\[|\mc{D}|\le N\sqrt{n}.\]
To summarize, we have shown that if $(X_1,\dots,X_n) \in \mc{E}_L$, then there exists a witnessing sequence $(t_i)_{i=0}^{\ell}$ with step record $(w_i)_{i=1}^{\ell}$ such that $t_0 \in\mc{D}$, and such that
\[\#\{i\in[\ell]: \text{step }i\text{ is }\lambda_{\ref{lem:entropy}}\text{-robust and is not an }R\text{-drop}\}\ge\gamma_{\ref{lem:entropy}} n.\]
Therefore, by the union bound, it follows that
\begin{align*}
    \mb{P}[\mc{E}_L] 
    &\le 8^n\sup_{\substack{I\subseteq[\ell], |I|=\lceil\gamma_{\ref{lem:entropy}}n\rceil\\t_0\in\mc{D}, (w_i)_{i=1}^\ell\in\{0,1\}^\ell_s}}\mb{P}[\text{The witnessing sequence starts at }t_0\text{, has step record }(w_i)_{i=1}^\ell\text{, and}\\
&\qquad\qquad\qquad\qquad\qquad\qquad\text{every }i\in I\text{ is }\lambda_{\ref{lem:entropy}}\text{-robust and is not an }R\text{-drop}].
\end{align*}

From \cref{eq:R-drop-bound}, taking $\delta_0 = \gamma_{\ref{lem:entropy}}/2$, it follows that the probability appearing on the right hand side above is bounded by
\[\bigg(\frac{8C_{\ref{lem:slice-levy-kolmogorov-rogozin}}(\lambda_{\ref{lem:entropy}}/2,\gamma_{\ref{lem:entropy}}/2,K_1,K_2)}{\gamma_{\ref{lem:entropy}}L}\bigg)^{\gamma_{\ref{lem:entropy}}n/2},\]
since there are at least  $\gamma_{\ref{lem:entropy}}n/2$ values of $i \in I$ with $i > \delta_0 n$ and since $R = \gamma_{\ref{lem:entropy}}L$ by definition. Therefore, taking $L$ and $n$ sufficiently large depending on $M$ and the parameters appearing above gives the desired conclusion.
\end{proof}

\subsection{Taming the preprocessed function}
In order to remove the dependence of $L_{\ref{prop:rough-Linfty}}$ on $M$, we will impose a log-Lipschitz condition on $f$. We refer the reader to \cite{Tik20} for an illuminating discussion of how such a regularity assumption is helpful. Given this log-Lipschitz assumption, \cite{Tik20} corrects the $\ell^{\infty}$ norm of a  preprocessed function (as in \cref{prop:rough-Linfty}) by showing that for any $\epsilon > 0$, an additional $\epsilon n$ amount of randomness (for $n$ sufficiently large) suffices to decrease the $\ell^{\infty}$-norm by a constant factor with superexponential probability. The way that this is accomplished is via a two-step process. First, one shows using an $\ell^{2}$-decrement argument that, after using an $\epsilon n/2$ amount of randomness, the density of points which achieve values within an absolute constant factor of the preprocessed $\ell^{\infty}$-norm is sufficiently small (depending on the log-Lipschitz constant of $f$). At this point, one can use an argument similar to \cref{prop:rough-Linfty} (in the independent case), replacing the application of the analogue of \cref{lem:slice-levy-kolmogorov-rogozin} with an estimate based on the `rarity of spikes'. The precise statement proved in \cite{Tik20} is \cref{prop:linfty-decrement}.

The main ingredient of this section is \cref{prop:slice-linfty-decrement}, which is a slice analogue of \cref{prop:linfty-decrement}. Remarkably, in order to prove \cref{prop:slice-linfty-decrement}, we will only need the statement of \cref{prop:linfty-decrement} as a black-box.

\begin{proposition}\label{prop:slice-linfty-decrement}
For any $p\in(0,1)$, $\epsilon\in(0,1)$, $\wt{R}\ge 1$, $L_0\ge 16\wt{R}$, and $M\ge 1$, there is $\gamma_{\ref{prop:slice-linfty-decrement}} = \gamma_{\ref{prop:slice-linfty-decrement}}(p) > 0$ and there are $n_{\ref{prop:slice-linfty-decrement}} = n_{\ref{prop:slice-linfty-decrement}}(p,\epsilon,L_0,\wt{R},M) > 0$ and $\eta_{\ref{prop:slice-linfty-decrement}} = \eta_{\ref{prop:slice-linfty-decrement}}(p,\epsilon,L_0,\wt{R},M)\in(0,1)$ with the following property. Let $L_0\ge L\ge 16\wt{R}$, let $n\ge n_{\ref{prop:slice-linfty-decrement}}$, $N\le 2^n$, and let $g\in\ell_1(\mb{Z})$ be a nonnegative function satisfying
\begin{itemize}
    \item $\snorm{g}_1 = 1$,
    \item $\log_2 g$ is $\eta_{\ref{prop:slice-linfty-decrement}}$-Lipschitz,
    \item $\sum_{t\in I}g(t)\le\wt{R}/\sqrt{n}$ for any integer interval $I$ of size $N$, and
    \item $\snorm{g}_\infty\le L/(N\sqrt{n})$.
\end{itemize}
For each $i\le 2\lfloor\epsilon n\rfloor$, let $Y_i$ be a random variable uniform on some disjoint union of integer intervals of cardinality at least $N$ each, and assume that $Y_1,\ldots,Y_{2\lfloor\epsilon n\rfloor}$ are mutually independent. Define a random function $\wt{g}\in\ell_1(\mb{Z})$ by
\[\wt{g}(t) = \mb{E}_b\bigg[g\bigg(t+\sum_{i=1}^{2\lfloor\epsilon n\rfloor}b_iY_i\bigg)\bigg|\sum_{i=1}^{2\lfloor\epsilon n\rfloor}b_i = s\bigg],\]
where $b_1,\dots, b_{2\lfloor \epsilon n \rfloor}$ are independent $\on{Ber}(p)$ random variables and  $s/(2\lfloor\epsilon n\rfloor)\in[p-\gamma_{\ref{prop:slice-linfty-decrement}},p+\gamma_{\ref{prop:slice-linfty-decrement}}]$. Then
\[\mb{P}\bigg[\snorm{\wt{g}}_\infty > \frac{19L/20}{N\sqrt{n}}\bigg]\le\exp(-Mn).\]
\end{proposition}
\begin{remark}
Note that $\gamma_{\ref{prop:slice-linfty-decrement}}$ depends only on $p$. 
\end{remark}

As mentioned above, we will crucially use the following $\ell^\infty$ decrement result for the independent model, proved in \cite{Tik20}.
\begin{proposition}[{\cite[Proposition~4.10]{Tik20}}]\label{prop:linfty-decrement}
For any $p\in(0,1/2]$, $\epsilon\in(0,1)$, $\wt{R}\ge 1$, $L_0\ge 16\wt{R}$, and $M\ge 1$ there are $n_{\ref{prop:linfty-decrement}} = n_{\ref{prop:linfty-decrement}}(p,\epsilon,L_0,\wt{R},M) > 0$ and $\eta_{\ref{prop:linfty-decrement}} = \eta_{\ref{prop:linfty-decrement}}(p,\epsilon,L_0,\wt{R},M)\in(0,1)$ with the following property. Let $L_0\ge L\ge 16\wt{R}$, let $n\ge n_{\ref{prop:linfty-decrement}}$, $N\le 2^n$, and let $g\in\ell_1(\mb{Z})$ be a nonnegative function satisfying
\begin{itemize}
    \item $\snorm{g}_1 = 1$,
    \item $\log_2 g$ is $\eta_{\ref{prop:linfty-decrement}}$-Lipschitz,
    \item $\sum_{t\in I}g(t)\le\wt{R}/\sqrt{n}$ for any integer interval $I$ of size $N$, and
    \item $\snorm{g}_\infty\le L/(N\sqrt{n})$.
\end{itemize}
For each $i\le\lfloor\epsilon n\rfloor$, let $Y_i$ be a random variable uniform on some disjoint union of integer intervals of cardinality at least $N$ each, and assume that $Y_1,\ldots,Y_{\lfloor\epsilon n\rfloor}$ are mutually independent. Define a random function $\wt{g}\in\ell_1(\mb{Z})$ by
\[\wt{g}(t) = \mb{E}_bg\bigg(t+\sum_{i=1}^{\lfloor\epsilon n\rfloor}b_iY_i\bigg)\]
where $b$ is a vector of independent $\on{Ber}(p)$ components. Then
\[\mb{P}\bigg[\snorm{\wt{g}}_\infty > \frac{(1-p(1-1/\sqrt{2}))L}{N\sqrt{n}}\bigg]\le\exp(-Mn).\]
\end{proposition}
\begin{remark}
We will only need this proposition for $p=1/2$. 
\end{remark}

\begin{proof}[Proof of \cref{prop:slice-linfty-decrement}]
As in the proof of \cref{lem:slkr}, we can use an equivalent method of sampling from the $s$-slice to rewrite $\wt{g}(t)$ as
\begin{align*}
\wt{g}(t) &= \mb{E}_b\bigg[g\bigg(t+\sum_{i=1}^{\lfloor\epsilon n\rfloor}(b_{2i-1}Y_{2i-1}+b_{2i}Y_{2i})\bigg)\bigg|\sum_{i=1}^{2\lfloor\epsilon n\rfloor}b_i = s\bigg]\\
&= \mb{E}_{b,b'}\bigg[g\bigg(t+\sum_{i=1}^{\lfloor\epsilon n\rfloor}(b_{2i-1}Y_{2i-1}+b_{2i}Y_{2i} + b_i'(b_{2i}-b_{2i-1})(Y_{2i-1}-Y_{2i})\bigg)\bigg|\sum_{i=1}^{2\lfloor\epsilon n\rfloor}b_i = s\bigg],
\end{align*}
where $b'$ is an $\lfloor \epsilon n \rfloor$-dimensional vector with \emph{independent} $\on{Ber}(1/2)$ components. Below, we will fix $b$ and use only the randomness in $b'$. In order to do this, let
\[B_0 := \bigg\{b_1,\dots, b_{2\lfloor \epsilon n \rfloor} : \#\{i : b_{2i-1}=1,b_{2i}=0\}\ge p(1-p)\epsilon n/4 \bigg\}.\]
Then, provided that $\gamma_{\ref{prop:slice-linfty-decrement}}$ is chosen sufficiently small depending on $p$, and $n$ is sufficiently large depending on $p$ and $\epsilon$, we have 
\[\mb{E}_b\bigg[1_{B_0} \bigg|  \sum_{i=1}^{2\lfloor \epsilon n \rfloor }b_i = s\bigg] > \frac{1}{2}.\]
Let $\mc{E}_L$ denote the event (depending on $Y_1,\dots, Y_{2\lfloor \epsilon n \rfloor}$) that $\snorm{\wt{g}}_\infty > 19L/(20N\sqrt{n})$. Now, suppose $Y_1,\dots, Y_{2\lfloor \epsilon n \rfloor} \in \mc{E}_L$, and suppose further that $\snorm{\wt{g}}_{\infty}$ is attained at $t \in \mb{Z}$.
Let 
\[B_{1} := \bigg\{b_1,\dots, b_{2\lfloor \epsilon n \rfloor}: \mb{E}_{b'}\bigg[g\bigg(t+\sum_{i=1}^{\lfloor\epsilon n\rfloor}(b_{2i-1}Y_{2i-1}+b_{2i}Y_{2i} + b_i'(b_{2i}-b_{2i-1})(Y_{2i-1}-Y_{2i})\bigg)\bigg|b\bigg]\ge\frac{9L/10}{N\sqrt{n}}\bigg\}.\]
Since $\snorm{g}_\infty\le L/(N\sqrt{n})$, it follows from the reverse Markov inequality that  
\[\mb{E}_b\bigg[1_{B_1} \bigg| \sum_{i=1}^{2\lfloor \epsilon n \rfloor }b_i = s\bigg] > \frac{1}{2}.\]
Thus, we see that for every $(Y_1,\dots, Y_{2\lfloor \epsilon n \rfloor}) \in \mc{E}_L$, there exists some $b \in B_0 \cap B_1$. Hence, taking a union bound, we see that
\begin{align}
\label{eq:union-B_0}
\mb{P}&\bigg[\snorm{\wt{g}}_\infty > \frac{19L/20}{N\sqrt{n}}\bigg] \le \mb{P}[\exists b \in B_0 : b\in B_1]
 \nonumber \\
&\le|B_0|\sup_{b\in B_0}\mb{P}\bigg[\exists t: \mb{E}_{b'}\bigg[g\bigg(t+\sum_{i=1}^{\lfloor\epsilon n\rfloor}(b_{2i-1}Y_{2i-1}+b_{2i}Y_{2i} + b_i'(b_{2i}-b_{2i-1})(Y_{2i-1}-Y_{2i}))\bigg)\bigg]\ge\frac{9L/10}{N\sqrt{n}}\bigg].
\end{align}
We now bound the probability appearing on the right hand side of the above equation uniformly for $b \in B_0$. We fix $b \in B_0$, and note that, by definition, there is a set $I = \{i_1,\dots, i_k\} \subseteq \lfloor \epsilon n \rfloor$ such that $|I| = k \ge p(1-p)\epsilon n/4$ and such that for all $j \in [k]$,
\[b_{2{i_j}-1}Y_{2{i_j}-1}+b_{2{i_j}}Y_{2{i_j}}+b_{i_j}'(b_{2{i_j}}-b_{2{i_j-1}})(Y_{2i_j-1}-Y_{2i_j}) = Y_{2i_j-1}+b_{i_j}'(Y_{2i_j}-Y_{2i_j-1}).\]
For $j \in [k]$, let $Y_j^{b}:= Y_{2i_j} - Y_{2i_j - 1}$. Let $Y_{-2\cdot I}$ denote all components of $Y_1,\dots,Y_{2\lfloor \epsilon n \rfloor}$, except those corresponding to indices in $2\cdot I$, and let $Y_{2\cdot I}$ denote the remaining components. Then, for $b \in B_0$ and a choice of $Y_{-2\cdot I}$, we define the random function (depending on $Y_{2\cdot I}$),
\[\wt{g}_{b,Y_{-2\cdot I}}(t) := \mb{E}_{b'}g\bigg(t+\sum_{j=1}^{\lfloor p(1-p)\epsilon n/4\rfloor}b_j'Y_j^b\bigg).\]
Thus, we see that for any $b \in B_0$ and $Y_{-2\cdot I}$, the probability appearing on the right hand side of \cref{eq:union-B_0} is bounded by
\[\mb{P}\bigg[\snorm{\wt{g}_{b, Y_{-2\cdot I}}}_{\infty} \ge \frac{9L/10}{N\sqrt{n}} \bigg],\]
where the probability is over the choice of $Y_{2\cdot I}$. 

At this point, we can apply \cref{prop:linfty-decrement} to $\wt{g}_{b,Y_{-2\cdot I}}$. Let us quickly check that the hypotheses of \cref{prop:linfty-decrement} are satisfied. The assumptions on $g$ needed in \cref{prop:linfty-decrement} are satisfied because the same properties are assumed in \cref{prop:slice-linfty-decrement} (see below for the log-Lipschitz condition). Moreover, $b_1',\dots, b'_{\lfloor p(1-p)\epsilon n/4 \rfloor}$ are independent $\on{Ber}(1/2)$ random variables. Finally, notice that, given $Y_{-2\cdot I}$, each $Y_j^{b}$ is a random variable uniform on some disjoint intervals of cardinality at least $N$ each (since $Y_j^{b}$ is a translation of $Y_{2i_j}$ which is assumed to satisfy this property).

Thus, \cref{prop:linfty-decrement} shows that the expression on the right hand side of \cref{eq:union-B_0} is bounded above by
\begin{align*}
|B_0|\sup_{b\in B_0, Y_{-2\cdot I} }\mb{P}\bigg[\snorm{\wt{g}_{b, Y_{-2\cdot I}}}_\infty > \frac{9L/10}{N\sqrt{n}}\bigg]\le 2^n\exp(-Mp(1-p) n/4),
\end{align*}
provided that we choose $\eta_{\ref{prop:slice-linfty-decrement}}$ sufficiently small compared to $\eta_{\ref{prop:linfty-decrement}}(1/2, p(1-p)\epsilon/8, L_0, \tilde{R}, M)$. The desired result now follows after rescaling $M$ by a constant factor (depending on $p$).
\end{proof}

\subsection{Putting everything together}
We now prove the main result of this section, \cref{thm:inversion-of-randomness}. 
\begin{proof}[Proof of \cref{thm:inversion-of-randomness}]
Fix any admissible parameters $\delta, K_1,K_2,K_3,p,\epsilon,N$, and the given parameter $M \ge 1$. We need to choose $L_{\ref{thm:inversion-of-randomness}}, \gamma_{\ref{thm:inversion-of-randomness}}, \eta_{\ref{thm:inversion-of-randomness}}, n_{\ref{thm:inversion-of-randomness}}$, where the first two quantities are allowed to depend on all the parameters \emph{except} M, and the last two quantities are allowed to depend on all the parameters. 

We let 
\[L' := L_{\ref{prop:rough-Linfty}}(p, \epsilon/2, \delta, K_1, K_2, K_3, M); \quad \gamma_{\ref{thm:inversion-of-randomness}} := \gamma =  \min\{\gamma_{\ref{prop:rough-Linfty}}(p,\epsilon/2, \delta, K_1,K_2,K_3), \gamma_{\ref{prop:slice-linfty-decrement}}(p)\}/4;\]
note that $\gamma_{\ref{thm:inversion-of-randomness}} = \gamma$ does not depend on $M$.
We choose
\[\wt{R}:= C_{\ref{lem:slice-levy-kolmogorov-rogozin}}(1/4, \delta/2, K_1, K_2); \quad L_{\ref{thm:inversion-of-randomness}}:= 16\wt{R};\]
note that $L_{\ref{thm:inversion-of-randomness}}$ does not depend on $M$, as desired. 
We choose $q$ to be the smallest positive integer for which
\[0.95^{q} L' \le 16\wt{R}.\]
Now, let 
\[\eta_{\ref{thm:inversion-of-randomness}} = \eta_{\ref{prop:slice-linfty-decrement}}(p, \gamma n/2q, \max\{L', 16\wt{R}\}, \wt{R}, M),\]
and suppose that $f\in \ell_1(\mb{Z})$ with $\snorm{f}_{1} = 1$, and that $\log_2 f$ is $\eta_{\ref{thm:inversion-of-randomness}}$-Lipschitz.

\textbf{Step 1: }Let $\ell := \lceil (1-\gamma)n \rceil$. Since $N\le\binom{n}{pn}\exp(-\epsilon n)$, it follows from \cref{prop:rough-Linfty} and the choice of parameters that, as long as $s/\ell \in [p-4\gamma, p+4\gamma]$, then for all sufficiently large $n$,
\[\mb{P}[\snorm{f_{\mc{A},s,\ell}}_\infty\ge L'(N\sqrt{n})^{-1}]\le\exp(-2Mn).\]
Let $\mc{E}_0$ be the event that $\snorm{f_{\mc{A},s,\ell}}_{\infty} < L'(N\sqrt{n})^{-1}$ simultaneously for all $s \in [p\ell - 4\gamma\ell, p\ell + 4\gamma \ell]$. Then, by the union bound, we see that 
\[\mb{P}[\mc{E}_0^c]\le n\exp(-2Mn).\]

\textbf{Step 2: }We split the interval $[\ell+1,n]$ into $q$ subintervals of size $\gamma n/q$ each, which we denote by $I_1,\dots, I_q$. Note that
\[f_{\mc{A}, s, n}(t) = \mb{E}_b f_{\mc{A}, s', \ell}\bigg(t + \sum_{i=1}^{q}\sum_{j\in I_i}b_jX_j\bigg);\]
here, we have sampled a uniform point in $\{0,1\}^{n}_{s}$ by first sampling its last $\gamma n$ coordinates, which we denote by $b = (b_{\ell+1},\dots,b_n)$, and then sampling the remaining coordinates, subject to the constraint that their sum is exactly $s':= s - b_{\ell+1}-\dots - b_n$.

Note that if $s \in [pn - \gamma n, pn + \gamma n]$, then we always have $s' \in [p\ell - 2\gamma\ell, p\ell + 2\gamma\ell]$. In particular, on the event $\mc{E}_0$, we have that for all $s \in [pn - \gamma n, pn + \gamma n]$, for all possible realizations of $s'$,
\[\snorm{f_{\mc{A},s',\ell}}_{\infty} < L'(N\sqrt{n})^{-1}.\]

\textbf{Step 3: }For $i \in [q]$, let
\[s_i := \sum_{j\in I_i}b_j,\]
where we use notation as above. Let $\mc{G}$ be the event that $s_i \in [p|I_i|-\gamma|I_i|, p|I_i|+\gamma|I_i|]$ for all $i\in [q]$. Then, by a standard large deviation estimate, it follows that for all $n$ sufficiently large, $\mb{P}[\mc{G}^{c}] \le n^{-1/2}$ (say). Hence, using the conclusion of the previous step, conditioning on the values $s_1,\dots, s_q$, and using the law of total probability, we see that on the event $\mc{E}_0$,
\begin{align}
\label{eq:sup-s_i}
f_{\mc{A},s,n}(t) \le L'n^{-1/2}\cdot (N\sqrt{n})^{-1} + \sup_{s_1/|I_1|,\dots, s_q/|I_q| \in [p-\gamma, p+\gamma]}\mb{E}_b f_{\mc{A}, s', \ell}\bigg(t + \sum_{i=1}^{q}\sum_{j\in I_i}b_jX_j\bigg),
\end{align}
where each vector $(b_j)_{j\in I_i}$ is independently sampled uniformly from $\{0,1\}^{I_i}_{s_i}$. Fix integers $s_1,\dots,s_q$ such that $|s_i|/|I_i| \in [p-\gamma, p+\gamma]$. In particular, this fixes $s' = s - s_1-\dots -s_q$. We define the sequence of functions $(g_k)_{k=0}^{q}$ by
\begin{align*}
    g_0(t) &= f_{\mc{A},s',\ell}(t)\\
    g_k(t) &= \mb{E}_b\bigg[f_{\mc{A},s',\ell}\bigg(t + \sum_{i=1}^{k}\sum_{j \in I_i}b_j X_j\bigg)\bigg| \sum_{j \in I_i}b_j = s_i \forall i \in [k] \bigg]\text{ for }k\ge 1.
\end{align*}
Note that for all $k \in [q]$,
\begin{align*}
    g_k(t) = \mb{E}_b\bigg[g_{k-1}\bigg(t + \sum_{j \in I_k}b_j X_j\bigg) \bigg| \sum_{j \in I_k}b_j = s_k\bigg].
\end{align*}

\textbf{Step 4: }We now wish to apply \cref{prop:slice-linfty-decrement} to $g_0,\dots, g_q$ successively with parameters $L_0,\dots, L_q$ given by
\[L_k := L'\cdot (19/20)^{k}.\]
More precisely, for $k\ge 1$, let $\mc{E}_k$ denote the event that $\snorm{g_k}_{\infty} \le L_k (N\sqrt{n})^{-1}$. We claim that 
\begin{align}
\label{eq:taming}
\mb{P}[\mc{E}_k | \mc{E}_{k-1}] \ge 1-\exp(-2Mn)\text{ for all } k \in [q].
\end{align}
Let us quickly check that on the event $\mc{E}_{k-1}$, the hypotheses of \cref{prop:slice-linfty-decrement} are satisfied for $g_{k-1}$. We have $\snorm{g_{k-1}} = 1$ and $\log_{2}g_{k-1}$ is $\eta_{\ref{thm:inversion-of-randomness}}$-Lipschitz since it is a convex combination of functions satisfying the same properties. The condition $\snorm{g_{k-1}}_{\infty} \le L_k (N\sqrt{n})^{-1}$ holds on $\mc{E}_{k-1}$ by definition. Moreover, the condition that for any interval $I$ of size $N$,
\[\sum_{t\in I}g_{k-1}(t) \le \wt{R}/\sqrt{n}\]
follows since, by \cref{lem:slice-levy-kolmogorov-rogozin} and our choice of $\wt{R}$, the analogous property holds for $f_{\mc{A},s',\ell}$, and hence for $g_{k-1}$, which is a convex combination of translates of $f_{\mc{A},s',\ell}$. \cref{prop:slice-linfty-decrement} now justifies \cref{eq:taming}. In particular, by the union bound, we have
\[\mb{P}[\mc{E}_q \mid  \mc{E}_0] \ge 1-q\exp(-2Mn).\]
Combining this with the estimate on $\mc{P}[\mc{E}_0^{c}]$, using \cref{eq:sup-s_i} (with an at most $n^{q}$ sized union bound to account for the choice of $s_1,\dots,s_q$), and finally taking $n$ sufficiently large gives the desired conclusion.
\end{proof}

\section{The structure theorem}
\label{sec:structure-theorem}
\begin{definition}
Let $p \in (0,1/2]$, $\gamma \in (0,p)$, and $L \ge 1$. Then, for any integer $n \ge 1$ and $x \in \mb{S}^{n-1}$, we define
\[\mc{T}_{p, \gamma}(x,L) := \sup\bigg\{t \in (0,1): \mc{L}_{p, \gamma}\bigg(\sum_{i=1}^{n}b_i x_i, t\bigg) > Lt\bigg\}. \]
\end{definition}

For $p \in (0,1/2)$, let $H = H(p)$ denote an $(n-1) \times n$ random matrix, each of whose entries is an independent copy of a $\on{Ber}(p)$ random variable. We fix a function $v(H)$ which takes as input an $(n-1)\times n$ matrix and outputs a unit vector in its right kernel. The goal of this section is to prove the following result about the threshold function of $v(H)$. 

\begin{proposition}
\label{prop:structure}
Let $\delta, \rho, \epsilon \in (0,1)$. There exist $L_{\ref{prop:structure}} = L_{\ref{prop:structure}}(\delta, \rho, p, \epsilon)$, $\gamma_{\ref{prop:structure}} = \gamma_{\ref{prop:structure}}(\delta, \rho, p, \epsilon)$ and $n_{\ref{prop:structure}} = n_{\ref{prop:structure}}(\delta, \rho, p, \epsilon)$ such that for all $n \ge n_{\ref{prop:structure}}$, with probability at least $1-4^{-n}$, exactly one of the following holds. 
\begin{itemize}
    \item $v(H) \in \on{Cons}(\delta, \rho)$, or
    \item $\mc{T}_{p, \gamma_{\ref{prop:structure}}}(v(H), L_{\ref{prop:structure}}) \le \binom{n}{pn}^{-1}\exp(\epsilon n)$.
\end{itemize}
\end{proposition}

The proof will make use of the following lemma, proved using randomized rounding (cf.~\cite{Liv18}), which is a slight generalization of \cite[Lemma~5.3]{Tik20}. As the proof is identical, we omit details. 
\begin{lemma}\label{lem:round}
Let $y = (y_1,\ldots,y_n)\in \mb{R}^n$ be a vector, and let $\mu>0$, $\lambda\in \mb{R}$ be fixed. Let $\Delta$ denote a probability distribution which is supported in $\{0,1\}^{n}$. There exist absolute constants $c_{\ref{lem:round}}$ and $C_{\ref{lem:round}}$ for which the following holds. 

Suppose that for all $t\ge \sqrt{n}$,
\[\mb{P}\bigg[\bigg|\sum_{i=1}^nb_iy_i-\lambda\bigg|\le t\bigg]\le \mu t,\]
where $(b_1,\dots, b_n)$ is distributed according to $\Delta$. Then, there exists a vector $y' = (y'_1,\dots,y'_n) \in \mb{Z}^{n}$ satisfying
\begin{enumerate}[(R1)]
    \item $\snorm{y-y'}_\infty\le 1$,
    \item $\mb{P}[|\sum_{i=1}^nb_iy_i'-\lambda|\le t]\le C_{\ref{lem:round}}\mu t$ for all $t\ge \sqrt{n}$,
    \item $\mc{L}(\sum_{i=1}^nb_iy_i',\sqrt{n})\ge c_{\ref{lem:round}}\mc{L}(\sum_{i=1}^nb_iy_i,\sqrt{n})$,
    \item $|\sum_{i=1}^ny_i-\sum_{i=1}^ny_i'|\le C_{\ref{lem:round}}\sqrt{n}$.
\end{enumerate}
\end{lemma}

\begin{proof}[Proof of \cref{prop:structure}]
For lightness of notation, we will often denote $v(H)$ simply by $v$. If $v \notin \on{Cons}(\delta, \rho)$, it follows from \cref{lem:slkr} that for all $\gamma < p/4$, there exists some $L_0 = L_0(\delta, \rho, p)$ and $C_0 = C_0(\delta, \rho, p)$ such that $\mc{T}_{p, \gamma}(v,L) \le C_0\cdot n^{-1/2}$ for all $L \ge L_0$. Fix $K_0$ such that the event $\{\snorm{H - pJ_{n-1\times n}} \le K_0 \sqrt{n}\}$ holds with probability at least $1-2^{-1729n}$.

Let $L > L_0$ be a parameter to be chosen later, depending on $\delta, \rho, p, \epsilon$. Let $\gamma < 1/4$ be a parameter to be chosen later depending on $\delta, \rho, p, \epsilon$. Fix some $N \in [C_0^{-1}\cdot \sqrt{n},\binom{n}{pn}\exp(-\epsilon n)]$, and let $\mc{U}_N$ denote the event that $\mc{T}_{p,\gamma}(v,L) \in [1/N, 2/N]$. We proceed to bound $\mb{P}[\mc{U}_N \wedge \mc{E}_{K_0}]$. 

Let $D := C_1 \sqrt{n} N$, where $C_1 = C_1 (\delta, \rho) \ge 1$ will be an integer chosen later. Let $y := Dv$. Since for all $t \ge \sqrt{n}$,
\begin{align*}
    \mb{P}\bigg[\bigg|\sum_{i=1}^{n} b_i y_i\bigg| \le t\bigg]
    &= \mb{P}\bigg[\bigg|\sum_{i=1}^{n} b_i v_i\bigg| \le \frac{t}{C_1 \sqrt{n}N}\bigg]\\
    &\le \mb{P}\bigg[\bigg|\sum_{i=1}^{n} b_i v_i\bigg| \le \frac{t}{\sqrt{n}N}\bigg]
    \le \frac{L}{N}\cdot \frac{2t}{\sqrt{n}}, 
\end{align*}
it follows that by applying \cref{lem:round} to $y$, with $\mu:= 2L/(N\sqrt{n})$, $\lambda = 0$, and the distribution on $\mb{R}^{n}$ coinciding with that of $n$ independent $\on{Ber}(p)$ random variables conditioned to have sum in $[pn -\gamma n, pn + \gamma n]$, we see that for all sufficiently large $n$, there exists some $y' \in \mb{Z}^{n}$ satisfying the conclusions of \cref{lem:round} (note that $C_{\ref{lem:round}}, c_{\ref{lem:round}}$ in this case are absolute constants). By (R3), we have
\begin{align}
\label{eq:lower-bound-levy}
    \mc{L}\bigg(\sum_{i=1}^{n}b_i y_i', \sqrt{n}\bigg) 
    &\ge c_{\ref{lem:round}}\mc{L}\bigg(\sum_{i=1}^{n}b_i v_i, 1/(C_1N)\bigg) \nonumber \\
    &\ge (2C_1)^{-1}\cdot c_{\ref{lem:round}}\mc{L}\bigg(\sum_{i=1}^{n} b_i v_i, 2/N\bigg) \nonumber \\
    &\ge (2C_{1})^{-1}\cdot c_{\ref{lem:round}}\cdot 2LN^{-1}.
\end{align}
Moreover, by (R1) and (R4), we have on the event $\mc{E}_{K_0}$ that
\begin{align}
\label{eq:bound-norm-rounding}
    \snorm{H y'}_{2} 
    &= \snorm{H(y'-y)}_{2} \nonumber \\
    &\le \snorm{(H-pJ_{n-1\times n})(y'-y)}_{2} + \snorm{pJ_{n-1\times n}(y'-y)}_{2} \nonumber \\
    &\le K_0' n,
\end{align}
where $K_0'$ is a constant depending only on $K_0$.

We claim that there is an absolute constant $C_2 > 0$ and a constant $C_3 = C_3(\delta, \rho)$, a collection of real numbers (depending on $\delta, \rho$) $(K_{3})_j > (K_{2})_j > (K_{1})_j > 1$ for $j \in [C_3\cdot C_2^{n}]$, a positive real number $\delta' > 0$  (depending on $\delta, \rho$), and a collection of $(N,n,(K_1)_j, (K_2)_j, (K_3)_j, \delta')$-admissible sets $\mc{A}_j$ (depending on $\delta, \rho$) for $j \in [C_3 \cdot C_2^{n}]$ such that $y' \in \mc{A}_j$ for some $j \in [C_3 \cdot C_2^{n}]$.  Let $\wt{v} := y'/D$. By (R1), it follows that $\snorm{\tilde{v} - v}_{\infty} \le D^{-1}$. Moreover, by \cref{lem:nonconstant-admissible}, there exist $\nu, \nu'$ depending on $\delta, \rho$, and a finite set $\mc{K}$ of positive real numbers, also depending on $\delta, \rho$, such that either the first conclusion or the second conclusion of \cref{lem:nonconstant-admissible} is satisfied for $v$. Since $D^{-1} \le C_{1}^{-1}C_0/n$, we see that there exists $n_0$ depending on $\delta, \rho, p$ such that for all $n \ge n_0$, $\wt{v}$ satisfies either the first conclusion or the second conclusion of \cref{lem:nonconstant-admissible}, with $\nu/2, \nu'/2$ and $2^{-1}\cdot \mc{K} \cup 2\cdot \mc{K}$. After paying an overall factor of at most $2^{n}$, we may assume that the $\nu n$ coordinates of $\wt{v}$ satisfying this conclusion are the first $\nu n$ coordinates. The remaining $(1-\nu)n$ coordinates of $\wt{v}$ lie in the $(1-\nu)n$-dimensional ball of radius $1$. By a volumetric argument, we see that this ball can be covered by at most $100^{n}$ translates of $[0, n^{-1/2}]^{(1-\nu)n}$. By paying an overall factor of $100^{n}$, we may fix the translate of $[0,n^{-1/2}]^{(1-\nu)n}$ that the remaining $(1-\nu)n$ coordinates lie in. Note that each such translate contains at most $(2D/\sqrt{n})^{(1-\nu)n}$ points in $(1/D)\mb{Z}^{n}$. Finally, taking $C_1(\delta, \rho)$ sufficiently large so that $C_{1}(\delta, \rho)\cdot \min(2^{-1}\cdot \mc{K}) > 1$ and rescaling by $D$ proves the claim. 

To summarize, we have so far shown the following. For parameters $L$ and $\gamma$ depending on $\delta, \rho, p, \epsilon$ (to be chosen momentarily), on the event $\mc{U}_N \wedge \mc{E}_{K_0}$, the event $\mc{B}_j$ holds for some $j \in [C_2\cdot C_3^{n}]$, where $\mc{B}_j$ is the event that there exists some $y' \in \mc{A}_j$ satisfying \cref{eq:lower-bound-levy}, \cref{eq:bound-norm-rounding}, and (by (R2)),
\begin{align}
\label{eq:sbp-y}
    \mb{P}\bigg[\bigg|\sum_{i=1}^{n}b_i y_i'\bigg| \le t\bigg] \le C_{\ref{lem:round}}\mu t \text{ for all }t\ge \sqrt{n},
\end{align}
where recall that $\mu = 2L/(N\sqrt{n})$. 

We are now ready to specify the parameters $L$ and $\gamma$. First, let
\[L' := \max_j L_{\ref{cor:threshold-inversion}}(p, \epsilon, \delta', (K_1)_j, (K_2)_j, (K_3)_j); \gamma' := \min_j \gamma_{\ref{cor:threshold-inversion}}(p, \epsilon, \delta', (K_1)_j, (K_2)_j, (K_3)_j)/4.\]
Then, let
\[L:= (2C_1)\cdot c_{\ref{lem:round}}^{-1}\cdot L' + L_0; \gamma:= \gamma'.\]

Our goal is to bound $\mb{P}[\cup_j \mc{B}_j]$. Let $H_1,\dots,H_{n-1}$ denote the rows of $H$. By a standard large deviation estimate, we can find an absolute constant $Q \ge 1$ such that the event 
\[\mc{W}_Q: = \{ |\{i \in [n-1]: \sum_{j=1}^{n}H_{i,j} \notin [pn - \gamma n, pn+\gamma n] \}| \le Q \} \]
holds with probability at least $1- 2^{-1729n}$. Then, it suffices to bound $\mb{P}[\cup_j (\mc{B}_j \wedge \mc{W}_Q)]$. We will provide a uniform (in $j$) upper bound on $\mb{P}[\mc{B}_j \wedge \mc{W}_Q]$, and then conclude using the union bound. Note that on the event $\mc{B}_j$, $y'$ belongs to the set $\mc{D}_j$ defined by
\[\mc{D}_j := \bigg\{x\in \mc{A}_j: \mc{L}_{p,\gamma}\bigg(\sum_{i=1}^nb_ix_i,\sqrt{n}\bigg)\ge L N^{-1}\bigg\}.\]
By the choice of $L$, it follows from \cref{cor:threshold-inversion} that for any $M \ge 1$, for all sufficiently large $n$,
\begin{align}
\label{eq:entropy}
|\mc{D}_j| \le e^{-Mn}|\mc{A}_j| \le e^{-Mn}(K_3N)^{n},
\end{align}
where $K_3:= \max_j (K_3)_j$. Moreover, it follows from \cref{eq:sbp-y} and the standard tensorization lemma (cf.~\cite[Lemma~3.2]{Tik20}) that 
\begin{align}
\label{eq:tensorization}
\mb{P}[\{\snorm{Hy'}_{2} \le K_0'n\} \wedge \mc{W}_Q] \le \bigg(\frac{C_4LK_0'}{N}\bigg)^{n-Q},
\end{align}
where $C_4 \ge 1$ is an absolute constant. Here, we have used that on the event $\mc{W}_Q$, the entries of at least $n-Q$ rows have sum in $[pn - \gamma n, pn+\gamma n]$.

Finally, from \cref{eq:entropy} and \cref{eq:tensorization}, we see that first taking $M$ to be sufficiently large (compared to various constants depending on $\delta, \rho, p, \epsilon$), and then taking $n$ sufficiently large, gives the desired conclusion.
\end{proof}

\section{Proof of \texorpdfstring{\cref{thm:main}}{Theorem 1.3}}\label{sec:proof}
We now have all the ingredients needed to prove \cref{thm:main}. The proof uses the insight from \cite{LT20} of exploiting the exponential gap between $\binom{n}{pn}$ and $(1-p)^{n}$ for $p < 1/2$ by using a `row boosting' argument to reduce to an anticoncentration problem on a well-conditioned slice.

\begin{proof}[Proof of \cref{thm:main}]
Throughout, we fix functions $x(A),y(A)$ which take as input a matrix $A$ and output (fixed, but otherwise arbitrary) right and left least singular unit vectors, respectively. Let $B = B_n(p)$ for simplicity. Fix $\epsilon > 0$ such that $\binom{n}{pn}\exp(\epsilon n) \le (1-p-\epsilon)^{n}$.

\textbf{Step 1: }By the work of Rudelson and Vershynin \cite{RV08}, there is some $c_p > 0$ so that for all $t > 2^{-2c_p n}$,
\[\mb{P}[s_n(B_n(p))\le t/\sqrt{n}]\le C_pt;\]
note that there is a slight complication since $\on{Ber}(p)$ is not centered, but this can be handled using standard techniques (see, e.g., \cite[Theorem~1.3]{JSS20smooth}) Therefore, it suffices to consider the case $t\le 2^{-2c_p n}$. 

\textbf{Step 2: }For $\delta, \rho \in (0,1)$, we define
\begin{align*}
\mc{E}_L(\delta, \rho) &= \{\exists y\in\on{Cons}(\delta,\rho): \snorm{y(B)^{T}B}_2\le 2^{-c_p n}\},\\
\mc{E}_R(\delta, \rho) &= \{\exists x\in\on{Cons}(\delta,\rho): \snorm{Bx(B)}_2\le 2^{-c_p n}\}.
\end{align*}
Applying \cref{prop:compressible} with $c_p > 0$, we find that there exist $\delta,\rho,\epsilon' > 0$ such that for all sufficiently large $n$,
\[\mb{P}[s_n(B)\le t/\sqrt{n}]\le 2n(1-p)^n+2(1-p-\epsilon')^n + \mb{P}[s_n(B)\le t/\sqrt{n}\wedge\mc{E}_L(\delta, \rho)^c\wedge\mc{E}_R(\delta, \rho)^c].\]
Here, we have used that the distribution of $B$ is invariant under transposition. 

\textbf{Step 3: } Let $\gamma = \gamma_{\ref{prop:structure}}(\delta, \rho, p, \epsilon)$. Let $W_\gamma \subseteq \{0,1\}^{n}$ denote the set of vectors $x \in \{0,1\}^{n}$ such that $\sum_{i=1}^{n} x_i \in [pn - \gamma n, pn+\gamma n]$.  As in the proof of \cref{prop:structure}, let $Q \ge 1$ be a constant such that the event 
\[\mc{W}_Q: = \{ |\{i \in [n]: B_i \notin W_\gamma| \le Q \} \]
holds with probability at least $1- 2^{-1729n}$. Then, it suffices to bound $\mb{P}[s_n(B)\le t/\sqrt{n} \wedge \mc{E}_L^{c} \wedge \mc{E}_R^{c} \wedge \mc{W}_Q]$, where for simplicity, we have omitted the parameters $\delta, \rho$ fixed in the previous step.

Let $B_1,\ldots,B_n$ denote the rows of $B$, and for simplicity, let $y = y(B)$. On the event that $s_n(B)\le t/\sqrt{n}$, we have 
\[\snorm{y_1B_1+\cdots+y_nB_n}_2\le t/\sqrt{n}.\]
Moreover, on the event $\mc{E}_L^c$, using \cref{lem:nonconstant-admissible}, there is a set $I \subseteq [n]$ such that $|I|\ge \nu n$ and such that for all $i\in I$, $|y_i| \ge \kappa/\sqrt{n}$, for some $\kappa:=\kappa(\delta, \rho) > 0$. In particular, since for any $i\in [n]$,
$\snorm{y_1 B_1 + \dots + y_n B_n}_{2} \ge |y_i|\on{dist}(B_i, H_i),$
where $H_i$ denotes the span of rows $B_1,\dots,B_{i-1},B_{i+1},\dots B_n$, it follows that 
\[\on{dist}(B_i, H_i) \le \frac{t}{\kappa}\text{ for all }i \in I.\]
Also, on the event $\mc{W}_Q$, there are at least $\nu n/2$ indices $i \in I$ such that $B_i \in W_\gamma$. Thus, we see that
\[\mb{P}[s_n(B) \le t/\sqrt{n} \wedge \mc{E}_L^{c} \wedge \mc{E}_R^{c} \wedge \mc{W}_Q] \le \frac{2}{\nu n}\sum_{i=1}^{n}\mb{P}[\on{dist}(B_i, H_i) \le t/\kappa \wedge \mc{E}_L^{c} \wedge \mc{E}_R^{c} \wedge B_i \in W_\gamma].\]

\textbf{Step 4: }By symmetry, it suffices to bound $\mb{P}[\mc{B}_1]$, where 
\[\mc{B}_1 := \on{dist}(B_1, H_1) \le t/\kappa \wedge \mc{E}_R^{c} \wedge B_1 \in W_\gamma.\]
Let $v(H_1)$ be a unit vector normal to $H_1$. Then, by \cref{prop:structure}, except with probability $4^{-n}$ (over the randomness of $H_1$), exactly one of the following holds. 
\begin{itemize}
    \item $v(H_1) \in \on{Cons}(\delta, \rho)$, or
    \item $\mc{T}_{p, \gamma}(v(H_1), L) \le \binom{n}{pn}^{-1}\exp(\epsilon n),$
\end{itemize}
where $L := L_{\ref{prop:structure}}(\delta, \rho, p, \epsilon)$. If the first possibility occurs, then $\mc{B}_1$ cannot hold, since then, $v(H_1) \in \on{Cons}(\delta, \rho)$ satisfies
\[\snorm{Bv(H_1)}_{2} = |\langle B_1, v(H_1)\rangle | \le \on{dist}(B_1, H_1) \le t/\kappa \le 2^{-2c_p n}/\kappa,\]
which contradicts $\mc{E}_R^{c}$ for all $n$ sufficiently large. Hence, the second possibility must hold. But then, using $\on{dist}(B_1, H_1) \ge |\langle B_1, v(H_1) \rangle|$, we have that (over the randomness of $B_1$),
\begin{align*}
    \mb{P}[\on{dist}(B_1, H_1) \le t/\kappa \wedge B_1 \in W_\gamma] 
    &\le \mb{P}[|\langle B_1, v(H_1)\rangle| \le t/\kappa \mid B_1 \in W_\gamma]\\
    &\le \frac{Lt}{\kappa} + \binom{n}{pn}^{-1}\exp(\epsilon n)\\
    &\le \frac{Lt}{\kappa} + (1-p-\epsilon)^{n}.
\end{align*}
This completes the proof. 
\end{proof}

\section{Singularity of random combinatorial matrices}\label{sec:additional-results}
In this section, we discuss the proof of \cref{thm:row-regular}. Given \cref{prop:structure}, by a similar argument as in the previous section, we see that the only additional ingredient required is the following estimate for invertibility on almost constant vectors. 

\begin{proposition}
\label{prop:qn-compressible}
For any $\epsilon > 0$, there exist $\delta, \rho, c, n_0$ depending on $\epsilon$ such that for all $n\ge n_0$,
\[\mb{P}\bigg[\inf_{x \in \on{Cons}(\delta, \rho)}\snorm{Q_n x}_{2} \le c\sqrt{n} \vee \inf_{y \in \on{Cons}(\delta, \rho)}\snorm{y Q_n}_{2} \le c\sqrt{n}\bigg] \le \bigg(\frac{1}{2} + \epsilon\bigg)^{n}.\]
\end{proposition}

We begin with the easier case of $\snorm{yQ_n}_{2}$. 
\begin{lemma}
\label{lem:qn-left-compressible}
For any $\epsilon > 0$, there exists $c, n_0$ depending on $\epsilon$ such that for all $n \ge n_0$ and for any $y \in \mb{S}^{n-1}$,
\[\mb{P}[\snorm{yQ_n}_{2} \le c\sqrt{n}] \le (1/2 + \epsilon)^{n}.\]
\end{lemma}
\begin{proof}
Without loss of generality, we may assume that $|y_1| \ge \dots \ge |y_n|$. We divide the proof into two cases depending on $|y_1|$. Let $\delta > 0$ be a constant to be chosen at the end of the proof. 

\textbf{Case I: }$|y_1| < \delta$. Note that any entry in the first $n/4$ columns of $Q_n$, conditioned on all the remaining entries in the first $n/4$ columns of $Q_n$, is distributed as $\on{Ber}(p)$ for some $p \in [1/3, 2/3]$. Moreover, by \cref{lem:LKR}, it follows that for independent random variables $\xi_1,\dots,\xi_n$, where $\xi_i \sim \on{Ber}(p_i)$ for some $p_i \in [1/3, 2/3]$,
\[\mc{L}(y_1 \xi_1 + \dots +y_n \xi_n, \delta) \le 3C_{\ref{lem:LKR}}\delta.\]
Therefore, a slight conditional generalization of the second part of the tensorization lemma \cite[Lemma~3.2]{Tik20} (which has the same proof) shows that
\[\mb{P}[\snorm{y Q_n}_{2} \le \delta \sqrt{n/8}] \le (20C_{\ref{lem:LKR}}\delta)^{n/8} \le (1/4)^{n},\]
provided that $\delta$ is chosen sufficiently small depending on $C_{\ref{lem:LKR}}$. 

\textbf{Case II: }$|y_1| \ge \delta$. Let $R_1,\dots,R_n$ denote the rows of $Q_n$. Then,
\begin{align*}
    \mb{P}[\snorm{yQ_n}_{2} \le \delta c\sqrt{n}] 
    &\le \sup_{R_2,\dots,R_n}\mb{P}[\snorm{y_1 R_1 + y_2 R_2 + \dots + y_n R_n}_{2} \le \delta c\sqrt{n} | R_2,\dots, R_n]\\
    &\le \sup_{v \in \mb{R}^{n}}\mb{P}[\snorm{R_1 - v}_{2} \le c\sqrt{n}] \le \binom{n}{n/2}^{-1}\binom{n}{2c^{2}n} \le \bigg(\frac{1}{2} + \epsilon\bigg)^{n}, 
\end{align*}
provided that $c > 0$ is chosen to be sufficiently small depending on $\epsilon > 0$. This completes the proof.  
\end{proof}

Next, we deal with the harder case of $\snorm{Q_nx}_{2}$. We will need the following analogue of \cite[Lemma~3.5]{Tik20}
\begin{lemma}\label{lem:slice-levy}
For any $\epsilon \in (0,1/8)$, there exist $\theta = \theta(\epsilon) > 0$ and $n_0$ depending on $\epsilon$ for which the following holds. For all $n \ge n_0$ and for all $x\in \mb{S}^{n-1}$ such that $|\sang{x, 1_n/\sqrt{n}}|\le 1/2$, we have 
\[\mc{L}(q\cdot x, \theta)\le 1/2+\epsilon,\]
where $q$ is distributed uniformly on $\{0,1\}^{n}_{n/2}$. 
\end{lemma}
\begin{proof}
Without loss of generality, we may assume that $|x_1| \ge \dots \ge |x_n|$. Again, we divide the proof into two cases depending on $|x_1|$. Let $\delta > 0$ be a constant to be chosen at the end of the proof. 

\textbf{Case I: }$|x_1| < \delta$. Let $\mu:= \mb{E}[q\cdot x]$ and $\sigma^{2}:= \on{Var}(q\cdot x)$. Since $\sang{x, 1_n} \le \sqrt{n}/2$, a direct computation shows that $\sigma^{2} \ge 3/16$. Moreover, a quantitative combinatorial central limit theorem due to Bolthausen \cite{Bol84} shows that the $L^\infty$ distance between the cumulative distribution function of $(q\cdot x - \mu)/\sigma$ and that of the standard Gaussian is at most $C\delta$, where $C$ is an absolute constant. Hence, for all $\delta$ sufficiently small, we have $\mc{L}(q\cdot x, \delta) \le 1/4$ whenever $|x_1| < \delta$.  

\textbf{Case II: }$|x_1| \ge \delta$. Let $\mc{G}$ denote the event (depending on $q$) that
\[(n-\epsilon^{2}n - 1)^{-1}\sum_{i=2}^{n-\epsilon^{2}n}q_i \in [1/2-\epsilon^{4}/2, 1/2 + \epsilon^{4}/2].\]
Then for all sufficiently large $n$, we have
\begin{align*}
    \sup_{r \in \mb{R}}\mb{P}[|q\cdot x - r| \le \theta]
    &\le \sup_{r\in \mb{R}}\mb{P}[|q\cdot x - r| \le \theta \wedge \mc{G}] + \mb{P}[\mc{G}^{c}]\\
    &\le \sup_{r \in \mb{R}}\mb{P}[|q\cdot x - r| \le \theta \wedge \mc{G}] + 2\exp(-\epsilon^{8}n/128),
\end{align*}
where the final inequality is by a standard large deviation estimate. It remains to control $\mb{P}[|q\cdot x - r|\le \theta \wedge \mc{G}]$. For this, fix any realization  $q':= (q_2,\dots, q_{n-\epsilon^{2}n})$ satisfying $\mc{G}$. Note that 
\[1/2 - 2\epsilon^{2} \le \inf_{q'\in \mc{G}}\mb{P}[q_1 = 0 \mid q']\le \sup_{q' \in \mc{G}}\mb{P}[q_1 = 0 \mid q']\le 1/2 + 2\epsilon^{2}.\]
Note also that, since $\sum_{i\ge n-\epsilon^{2}n}x_i^{2} \le \epsilon^{2}$ (this uses $\snorm{x}_{2} = 1$ and $|x_1| \ge |x_2| \ge \dots \ge |x_n|$), it follows that 
\[\sup_{q' \in \mc{G}, q_1}\on{Var}\bigg[\sum_{i \ge n-\epsilon^{2}n} q_i x_i \bigg| q_1,q'\bigg] \le \epsilon^{2},\]
so that by Markov's inequality, 
\[\sup_{q' \in \mc{G}, q_1}\mb{P}\bigg[\bigg|\sum_{i \ge n-\epsilon^{2}n} q_i x_i - f(q',q_1)\bigg| \ge \frac{\delta}{8}\bigg| q',q_1\bigg] \le \frac{32\epsilon^{2}}{\delta^{2}},\]
where $f(q',q_1)$ denotes the mean of $\sum_{i \ge n-\epsilon^{2}n} q_i x_i$ conditioned on $q',q_1$.
Finally, since $|x_1| \ge \delta$, and since 
\[\sup_{q' \in \mc{G}}|f(q', 0) - f(q', 1)| \le |x_{n-\epsilon^{2}n}|\le 2/\sqrt{n},\] \
it follows by putting everything together that
\begin{align*}
    \sup_{r\in \mb{R}}\mb{P}[|q\cdot x - r| \le \theta \wedge \mc{G}]
    &\le \sup_{r\in \mb{R}}\sup_{q' \in \mc{G}}\mb{P}[|q\cdot x - r| \le \theta \mid q'] \le 1/2 + 2\epsilon^{2} + 64\epsilon^{2}/\delta^{2},
\end{align*}
provided that $\theta$ is chosen sufficiently small compared to $\delta$, and $n$ is sufficiently large. Indeed, the two values of $q_1 x_1$ (for $q_1 = 1$ and $q_1 = 0$) differ by $|x_1|$, which is at least $\delta$ by assumption, and the above discussion shows that given $q'$ and $q_1$, $\sum_{i\ge n- \epsilon^2n}q_i x_i$ is localized in an interval of length $\delta/2 + 2/\sqrt{n}$ except with probability at most $32\epsilon^{2}/\delta^{2}$. Since $\delta$ is an absolute constant coming from \textbf{Case I}, this gives the desired conclusion for all sufficiently small $\epsilon$, which completes the proof.
\end{proof}

Given the previous two lemmas, the proof of \cref{prop:qn-compressible} is by now standard.
\begin{proof}[Proof of \cref{prop:qn-compressible}]
The estimate for $\inf_{y \in \on{Cons}(\delta, \rho)}\snorm{yQ_n}_{2} \le c\sqrt{n}$ (for a suitable choice of $\delta, \rho, c$) follows immediately by combining \cref{lem:qn-left-compressible} with the low metric entropy of $\on{Cons}(\delta, \rho)$. To exploit the latter, one could either use a  randomized rounding based net construction due to Livshyts \cite[Theorem~4]{Liv18}, which uses that $\snorm{Q_n}_{\on{HS}}^{2} \le n^{2}$, or one could use the fact that there exists a constant $K$ such that with probability at least $1 - 4^{-n}$, all singular values of $Q_n$ except for the top singular value are at most $K\sqrt{n}$ (see \cite[Proposition~2.8]{Tra20}).

For the estimate on $\inf_{x \in \on{Cons}(\delta, \rho)}\snorm{Q_n x}_{2} \le c\sqrt{n}$ (for suitable $\delta, \rho, c$), we begin by using the fact \cite[Proposition~2.8]{Tra20} noted above that there exists a constant $K > 0$ such that with probability at least $1-4^{-n}$, the operator norm of $Q_{n}$ restricted to the subspace perpendicular to $1_{n}$ is at most $K\sqrt{n}$. Let us denote this event by $\mc{E}_K$. Then, on $\mc{E}_K$, for any $x \in \mb{S}^{n-1}$ such that $\sang{x, 1_n/\sqrt{n}} \ge 1/2$, we have $\snorm{Q_{n} x}_{2} \ge n/4 - K\sqrt{n}$.
Hence, on the event $\mc{E}_K$, and for all $n$ sufficiently large, it suffices to consider the infimum over those vectors $x \in \on{Cons}(\delta, \rho)$ which also satisfy $\sang{x, 1_n/\sqrt{n}} < 1/2$. For this, we can use \cref{lem:slice-levy-kolmogorov-rogozin} followed by the tensorization lemma (cf.~\cite[Lemma~3.2]{Tik20}), and then exploit the low metric entropy of $\on{Cons}(\delta, \rho)$ as above. We leave the details to the interested reader. 
\end{proof}


\bibliographystyle{amsplain0.bst}
\bibliography{main.bib}

\end{document}